\documentclass[11pt,a4paper]{article}

\usepackage[headinclude]{typearea}
\usepackage[english]{babel}           
\usepackage[T1]{fontenc}
\usepackage{scrtime}
\usepackage{amsmath,amsthm,amsfonts,amssymb}
\usepackage {color}                                  
\usepackage {pifont}                                 
\usepackage {multicol,multirow}                               
\usepackage[numbers]{natbib}                                                     
\usepackage[normalem]{ulem}                                                     
\usepackage [scaled=.90]{helvet}                     
\usepackage {courier}                                
\usepackage {graphicx}                               
\usepackage {array}                                  
\usepackage {longtable}                              
\usepackage{fancyvrb}
\usepackage{enumerate} 
\usepackage{ifpdf}
\usepackage{caption}
\usepackage{mathrsfs}

\usepackage{setspace}

\usepackage{marvosym}

\usepackage{mathtools}
\mathtoolsset{showonlyrefs} 

\newtheorem{theorem}{Theorem}[section]

\newtheorem{lemma}[theorem]{Lemma}
\newtheorem{corollary}[theorem]{Corollary}

\theoremstyle{definition}
\newtheorem{assumption}{Assumption}[section]

\newtheorem{remark}[theorem]{Remark}

\newcommand{\Ito}{It\={o}{}}

\newcommand{\exclude}[1]{}

\newcommand{\densityconst}{{c_{\kappa}^{(p)}}}

\newcommand{\1}{\mathbf{1}}

\newcommand{\E}{{\mathbb{E}}}

\newcommand{\N}{{\mathbb{N}}}
\renewcommand{\P}{{\mathbb{P}}}
\newcommand{\Q}{{\mathbb{Q}}}
\newcommand{\R}{{\mathbb{R}}}

\newcommand{\sign}{\operatorname{sign}}

\definecolor{darkgreen}{rgb}{0,0.5,0}
\definecolor{lightgreen}{rgb}{0.5,0.9,0.5}
\definecolor{magenta}{rgb}{0.75,0,0.25}

\definecolor{violet}{rgb}{0.25,0,0.75}

\newcommand{\es}{{\underline{s}}}
\newcommand{\et}{{\underline{t}}}

\renewcommand{\P}{{\mathbb P}}

\newcommand{\cF}{{\cal F}}

\newcommand{\be}{\begin{equation}}
\newcommand{\Ee}{\end{equation}}
\newcommand{\bea}{\begin{eqnarray}}
\newcommand{\Ea}{\end{eqnarray}}
\newcommand{\beast}{\begin{eqnarray*}}
\newcommand{\East}{\end{eqnarray*}}
\newcommand{\bproof}{\begin{proof}}
\newcommand{\eproof}{\end{proof}}

\usepackage[numbers]{natbib}      
\title{The Euler-Maruyama Scheme for SDEs with Irregular Drift: Convergence Rates via Reduction to a Quadrature Problem.}
\author{Andreas Neuenkirch\footnote{Institut f\"ur Mathematik, Universit\"at Mannheim, B6, 26, D-68131 Mannheim, Germany,
{\tt neuenkirch@math.uni-mannheim.de }}
\and
Michaela Sz\"olgyenyi\footnote{Department of Statistics, University of Klagenfurt, Universit\"atsstra\ss{}e 65--67,
9020 Klagenfurt, Austria
{\tt 	michaela.szoelgyenyi@aau.at}}}

\date{Preprint, January 2020}

\begin{document}

\maketitle
\newcommand{\slugmaster}{}

\begin{abstract} We study the strong convergence order of the Euler-Maruyama scheme for scalar  stochastic differential equations  with additive noise and irregular drift. We provide a general framework for the error analysis  by reducing it to  a weighted quadrature problem for irregular functions of Brownian motion.
Assuming Sobolev-Slobodeckij-type regularity of order $\kappa \in (0,1)$ for the non-smooth part of the drift, our analysis of
the quadrature problem yields the convergence order $\min\{3/4,(1+\kappa)/2\}-\epsilon$ for the equidistant Euler-Maruyama scheme (for arbitrarily small $\epsilon>0$). The cut-off of the convergence order at $3/4$  can be overcome by using a suitable non-equidistant discretization, which yields the strong convergence order of $(1+\kappa)/2-\epsilon$ for the corresponding Euler-Maruyama scheme.\\

\end{abstract}

\pagestyle{myheadings}
\thispagestyle{plain}

\section{Introduction and Main Results}

Let $(\Omega, \cF, (\cF_t)_{t\in[0,T]}, \P)$ be a filtered probability space, where the filtration satisfies the usual conditions
and let  $W=(W_t)_{t \in [0,T]}$ be  a standard Brownian motion adapted to $(\cF_t)_{t\in[0,T]}$.
We consider  \Ito-stochastic differential equations (SDEs) of the form
\begin{align}\label{eq:sde}
X_t=\xi+\int_0^t \mu(X_s) ds +  W_t, \quad t \in[0,T],
\end{align}
where $T\in(0,\infty)$,   the drift coefficient $\mu\colon\R \rightarrow \R$ is  measurable and bounded, and the initial condition $\xi$ is independent of $W$.
Existence and uniqueness of a strong solution $X=(X_t)_{t\in[0,T]}$ to \eqref{eq:sde} is provided, e.g., in  \cite{zvonkin1974}.

For $n\in\N$ let $x^{(\pi_n)}=(x^{(\pi_n)}_t)_{t\in[0,T]}$ be the continuous-time Euler-Maruyama (EM) scheme
based on the discretization $$\pi_n=\{t_0, t_1, \ldots, t_n \} \qquad  \textrm{with} \qquad
 0=t_0 <t_1 < \ldots < t_n=T,$$ i.e.
\begin{align} x_t^{(\pi_n)}=\xi + \int_0^t \mu(x_{\es}^{(\pi_n)}) ds + W_t, \qquad t \in [0,T],  \label{euler} \end{align}
where
$ \et=\max \{t_{k}: \, t_{k} \leq t \}$. 
Our goal is to analyse the $L^2$-approximation error at the discretization points $t_k$, that is
\begin{align}  \max_{k\in\{0, \ldots, n\}}  \left( \E  \left[ \left|X_{t_k}-x_{t_k}^{(\pi_n)} \right|^2\right] \right)^{\!1/2}, \label{L2-error}
\end{align}
and in particular its dependence on $n$, i.e.~the scheme's convergence order. For this, we will study the time-continuous EM scheme and
\begin{align}  \sup_{t \in [0,T]}  \left( \E  \left[ \left|X_{t}-x_{t}^{(\pi_n)} \right|^2\right] \right)^{\!1/2}, 
\end{align}
which yields an upper bound for \eqref{L2-error}.

The error analysis of EM-type schemes for SDEs with discontinuous drift coefficient  has become -- after two pioneering articles  by \citet{gyongy1998} and  \citet{halidias2008}  --  a topic of growing interest in the recent years.

Articles which explicitly deal with the EM scheme for SDEs with irregular drift coefficients and additive noise  are
 \cite{halidias2008,lux,MENOUKEUPAMEN,gerencser2018}. Here, the best known results are from \citet{gerencser2018}: $L^2$-order $1/2-\epsilon$  for arbitrarily small $\epsilon>0$ is obtained for bounded and Dini-continuous drift coefficients for $d$-dimensional SDEs, while in the scalar case  one has $L^2$-order $1/2-\epsilon$ even for drift coefficients, which are only  bounded and integrable over $\mathbb{R}$.

For approximation results on SDEs with discontinuous drift coefficients and non-additive noise
see, e.g., \cite{sz15,ngo2016,sz2017a,ngo2017a,ngo2017b,sz2018a, sz2018b, muellergronbach2019}.
The best known results for EM schemes in this framework are
 $L^2$-order $1/2-\epsilon$  of an EM scheme with adaptive time-stepping for  multidimensional SDEs with piecewise Lipschitz drift and possibly degenerate diffusion coefficient, see \citet{sz2018b}, and 
$L^p$-order $1/2$ of the EM scheme for scalar SDEs with piecewise Lipschitz drift and possibly degenerate diffusion coefficient, see \citet{muellergronbach2019}.

Recently, also a transformation-based Milstein-type scheme has been analyzed for scalar SDEs by \citet{muellergronbach2019b}. They obtain  $L^p$-order $3/4$   for drift coefficients, which are piecewise  Lipschitz with piecewise Lipschitz derivative, and possibly degenerate diffusion coefficient.

Lower error bounds for the strong approximation of scalar SDEs with possibly discontinuous drift coefficients have been studied  in   \citet{hefter2018}. Assuming smoothness  of the coefficients only locally in a small neighbourhood of the initial value, the authors obtain for arbitrary methods that use a finite number of evaluations of the driving Brownian motion a  lower error bound of order one for the pointwise $L^1$-error. Lower  bounds will be also addressed in a forthcoming work by  \citet{muellergronbach2020}.

\smallskip

We will spell out a general framework for the analysis of the scheme \eqref{euler} for the SDE \eqref{eq:sde} under the following assumptions:
\begin{assumption}\label{ass} \quad  Assume that $\mu : \mathbb{R} \rightarrow \mathbb{R}$ with $\mu \neq 0$ can be decomposed into a regular and an irregular part $a, b\colon \mathbb{R} \rightarrow \mathbb{R}$, that is $\mu=a+b$, such that:
\begin{enumerate}
\item\label{ass-bd} (boundedness) \, $a, b: \mathbb{R} \rightarrow \mathbb{R}$ are bounded,
\item\label{ass-a}  (regular  part) \,\ $a \in C_b^2( \mathbb{R})$, i.e.~$a$  is twice continuously differentiable with bounded derivatives,
\item\label{ass-bl1}  (irregular part) \,$b \in L^1(\mathbb{R})$.
\end{enumerate} Moreover, we assume that
\begin{enumerate}
\item[(iv)] \label{ass-anfang} (initial value)  \, $\xi \in L^2(\Omega, \mathcal{F}_0, \P)$.
\end{enumerate}
\end{assumption}
\begin{assumption}\label{ass_b} \quad
 There exists  $\kappa \in  (0,1)$ such that 
$$  |b|_{\kappa}:= \left( \int_{\mathbb{R}} \int_{\mathbb{R}} \frac{|b(x)-b(y)|^2}{|x-y|^{2 \kappa  + 1}} \; dx \; dy \right)^{1/2}< \infty.$$
\end{assumption}
We call $|\cdot|_\kappa$ Sobolev-Slobodeckij semi-norm. Note that the decomposition of $\mu$ is only required for the error analysis and not for the actual implementation of the scheme.

\medskip

Assumption \ref{ass} is required for our perturbation analysis, where we use a suitable transformation of the state space and a Girsanov transform to show that  for all $\varepsilon\in(0,1)$ there exists a constant $C^{(R)}_{\varepsilon, a,b,T}>0$ such that
\begin{align} \label{intro-Pertub}   \sup_{t \in [0,T]} \E  \left[ \left|X_t-x_t^{(\pi_n)} \right|^2\right]  \leq  C^{(R)}_{\varepsilon, a,b,T} \cdot  \left( \|\pi_n\|^2 + \sup_{t \in [0,T]}|\mathcal{W}_t^{(\pi_n)}|^{1-\varepsilon} \right), \end{align}
where 
$$ \|\pi_n\|:=\max_{k= 0,\dots,n-1}|t_{k+1}-t_k| $$
and
\begin{align}
\mathcal{W}^{(\pi_n)}_{t} &= \E\!\left[ \left|  \int_{0}^{t}   \exp\left(-2\int_0^{W_s+\xi} b(z)dz \right)  \left[b(W_s+\xi)-  b(W_{\es}+\xi)\right] ds \right|^2\right], \quad t \in [0,T],
\end{align} 
see  Theorem \ref{main_1}.  The term  $\mathcal{W}^{(\pi_n)}_t$ corresponds to the  error of a quadrature problem, see Remark \ref{quad_rem}.

We would like to point out that 
\begin{itemize}
\item  this result provides a unifying general framework for the error analysis of the Euler-Maruyama scheme for SDEs with additive noise,
\item which can be used to analyse the convergence behaviour of the Euler-Maruyama scheme under very general assumptions on the drift coefficient  by various means for various discretizations.
\end{itemize}

We assume Sobolev-Slobodeckij regularity of order $\kappa \in (0,1)$ for $b$, i.e. Assumption \ref{ass_b}, and estimate  $\mathcal{W}^{(\pi_n)}$  for two different discretizations. For 
an equidistant discretization $\pi_n^{equi}$ given by
$$ t_k^{equi}= T \frac{k}{n}, \qquad  k=0, \ldots, n,$$
we obtain that $\mathcal{W}^{(\pi_n^{equi})}_t$ is of order $\min\{3/2,1+\kappa\}$ uniformly in $t \in [0,T]$
and consequently we have
\begin{align}  \label{eb-intro-1}
\sup_{t \in [0,T]} \left( \E\!  \left[ \left|X_t-x_t^{(\pi_n^{equi})} \right|^2\right] \right)^{\!1/2} \leq  C^{(EM),equi}_{\epsilon, \mu,T,\kappa} \cdot \left(   \frac{1}{n^{(1+\kappa)/2-\epsilon}} + \frac{1}{n^{3/4-\epsilon}} \right)
\end{align}
for $\epsilon >0$ arbitrarily small and a constant   $C^{(EM),equi}_{\epsilon, \mu,T,\kappa}>0$, independent of $n$, see Theorem \ref{main_2} and
Corollary \ref{cor_em}.
To overcome the cut-off of the convergence order for  $\kappa=1/2$, we use  a non-equidistant discretization $\pi_n^{*}$ given by
$$
 t_k^*=T  \left(\frac{k}{n}\right)^{\!2}, \qquad k=0, \ldots, n. $$
Similar non-equidistant nets have been used, e.g., in  \cite{lyons} to deal with weak error estimates for non-smooth functionals and in \cite{geiss} to deal with hedging errors in the presence of non-smooth pay-offs. We obtain that $\mathcal{W}^{(\pi_n^*)}_t$ is up to a log-term of order $1+\kappa$ uniformly in $t \in [0,T]$ and therefore we have
\begin{align}  \label{eb-intro-2} \sup_{t \in [0,T]} \left( \E\!  \left[ \left|X_t-x_t^{(\pi_n^{*})} \right|^2\right]\right)^{\!1/2}  \leq  C^{(EM),*}_{\epsilon, \mu,T,\kappa} \cdot  \frac{1}{n^{(1+\kappa)/2-\epsilon}} \end{align}
for $\epsilon >0$ arbitrarily small and a constant   $C^{(EM),*}_{\epsilon, \mu,T,\kappa}>0$, independent of $n$, see Theorem \ref{main_2} and
Corollary \ref{cor_em}.

\begin{remark}
\begin{enumerate}[(i)]
\item Our set-up covers a wide range of irregular perturbations. In particular, the use of  Sobolev-Slobodeckij regularity
allows to study irregular parts $b$ that are discontinuous. Examples include  indicator functions with compact support  or, more generally,  piecewise  H\"older continuous functions 
with compact support. In the former case one has Sobolev-Slobodeckij regularity of all orders $\kappa <1/2$, while for 
piecewise $\gamma$-H\"older continuous functions 
with compact support one has Sobolev-Slobodeckij regularity of all orders $\kappa < \min \{ 1/2, \gamma\}.$
Moreover functions, which are $\gamma$-H\"older continuous and have compact support, have  Sobolev-Slobodeckij regularity of all orders $\kappa < \gamma$.

Note that  Assumptions 1.1 and 1.2 imply that $b \in H^{\kappa}_{2}$, where $H^s_p$ with $s\in(0,\infty)$, $p\in[1,\infty)$ denotes the classical fractional Sobolev space, see, e.g., \cite{Sickel}.
Working in  $H^{s}_{p}$ or  in the Besov space $B^{s}_{p,q}$, where $q\in[1,\infty)$, could help to clarify the phenomenon why the same convergence order $3/4-\epsilon$ is 
obtained
for $\gamma$-H\"older continuous drift coefficients with $\gamma=1/2$ and for indicator functions as drift.

\item Our assumptions  cover also step functions as drift, i.e.
\begin{align} \label{step} \mu(x)=\sum_{\ell=1}^{L} \gamma_{\ell} \cdot \sign(x-x_i), \qquad x \in \mathbb{R}, \end{align}
with $L\in\N$, $\gamma_{1}, \ldots, \gamma_{L} \in \mathbb{R}$, and $-\infty <x_1 <x_2 < \ldots < x_L<\infty$.
This can be seen from the following: let $\mu(x)=\sign(x)$ and $\alpha\in(0,\infty)$. Then  the decomposition $a_\alpha,b_\alpha\colon \R\to\R$, $\mu(x)=a_\alpha(x)+b_\alpha(x)$, which  satisfies Assumption \ref{ass}  and Assumption \ref{ass_b} for all $\kappa <1/2$ and all $\alpha\in(0,\infty)$, can be chosen as
  \begin{align*}
   a_\alpha(x)=\begin{cases}
      1, & x\in(\alpha,\infty),\\
  \frac{2\int_\alpha^{\frac{x+3\alpha}{2}} (2\alpha-y)^2(y-\alpha)^2 dy}{\int_\alpha^{2\alpha} (2\alpha-y)^2(y-\alpha)^2 dy}-1, & x\in(-\alpha,\alpha),\\
    -1, & x\in(-\infty,-\alpha),
  \end{cases}
 \end{align*}
 and $b_\alpha(x)=\1_{(0,\alpha)}(x) \cdot (1-a_\alpha(x))+\1_{(-\alpha,0)}(x) \cdot (1-a_\alpha(x))$.
 Figure \ref{fig:signum} illustrates this decomposition.
 \begin{figure}[ht]
\begin{center}
\includegraphics[scale=0.4]{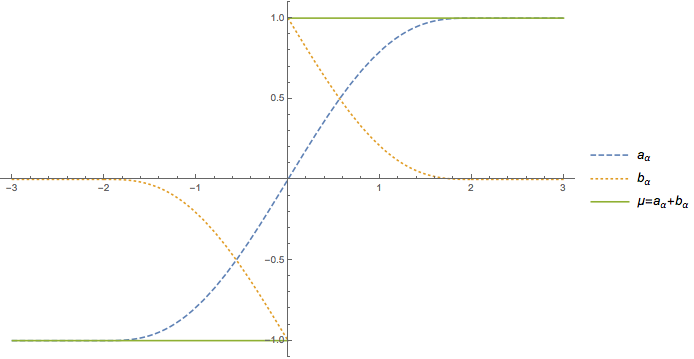}
\caption{A decomposition of the sign function ($\alpha=2$).}\label{fig:signum}
\end{center}
\end{figure}  Recall that such a decomposition of $\mu$ is only required for the error analysis and not for the actual implementation of the scheme.

\item In particular  for bounded $C_b^2(\mathbb{R})$-drift coefficients, which are perturbed by a step function  \eqref{step}, we obtain convergence order $3/4-\epsilon$ for all $\epsilon>0$, similar to the transformation-based Milstein-type method in \citet{muellergronbach2019b}. Moreover, for Lipschitz-continuous drift coefficients with bounded support we obtain convergence order $1-\epsilon$ for all $\epsilon>0$, similar to the drift-randomized Milstein-type scheme analyzed in \citet{raphael_2} under structurally different assumptions on the coefficient.

\item The reduction of the error of the EM scheme to a quadrature problem, i.e.~Theorem \ref{main_1}, relies among other results  on a Zvonkin-type transformation, see \cite{zvonkin1974}. 
For the analysis of numerical methods of SDEs with irregular coefficients this transformation has already been used, e.g., by \citet{ngo2017a}, and also the results of  \citet{MENOUKEUPAMEN,gerencser2018} rely on similar transformations.
In contrast to these works, we first split the drift-coefficient into a smooth and an irregular part, thus allowing a larger class of coefficients, and state with Theorem  \ref{main_1} a general reduction
result that explicitly links the error analysis of the EM scheme to the analysis of  quadrature problems.

\item Extensive numerical tests of the Euler scheme for different step functions as drift have been carried out  in \cite{lux}. In the absence of exact reference solutions, the estimates of the convergence rates via standard numerical tests turn out to be unstable and seem to depend  on the initial value and the fine structure of the step functions. For example, for $\xi=0$, $\mu=-\sign$  much better convergence rates are obtained than for $\xi=0$, $\mu=\sign$, although the Sobolev-Slobodeckij regularity remains unchanged.In particular, in some cases the estimated convergence orders are much worse than the guaranteed order $3/4-\epsilon$, which illustrates the unreliability of standard tests for such equations.

	\item In order to extend our result to the multidimensional case, we would need a multidimensional version of the Zvonkin-type transformation that we use here. A candidate for this would be a Veretennikov-type transformation, see  \cite{veretennikov1984}. However, this transformation is not given explicitly, but as solution to a PDE, and also other favourable properties are lost. Hence, the extension to the multidimensional case is out of the scope of the current paper as well as an extension to the Euler-Maruyama scheme for scalar SDEs with non-additive noise. While Zvonkin's transformation is still available, the Girsanov technique from Section \ref{sec:red_quad} is not applicable in this case due to the non-constant diffusion coefficient.

\end{enumerate}
\end{remark}

\begin{remark} \label{ytm}
Lamperti's transformation, i.e.
$$ \lambda\colon\R\to\R,  \qquad \lambda(x)= \int_{x_0}^x \frac{1}{\sigma(z)} dz,$$
with $x_0\in\R$,
reduces general scalar SDEs
$$ dX_t= \mu(X_t) dt + \sigma(X_t) dW_t,  \quad t \in [0,T], \qquad X_0=x_0,$$ with sufficiently smooth elliptic diffusion coefficient $\sigma\colon \mathbb{R} \rightarrow \mathbb{R}$ to SDEs of the form
$$ dY_t= g(Y_t) dt + dW_t, \quad t \in [0,T], \qquad
Y_0 =  \lambda(x_0),$$
with additive noise, where
$$  g(x)= \frac{\mu(\lambda^{-1}(x))}{\sigma(\lambda^{-1}(x))} - \frac{1}{2}\sigma'(\lambda^{-1}(x)), \qquad x \in \mathbb{R},  $$
and $X(t)= \lambda^{-1}(Y(t))$, $t \in [0,T]$.
If $\mu$ satisfies Assumptions \ref{ass} and \ref{ass_b} and   if $\sigma$ is three times continuously differentiable with bounded derivatives and
$$ 0< \inf_{x \in \mathbb{R}} \sigma(x) \leq  \sup_{x \in \mathbb{R}} \sigma(x) < \infty,$$
then $g$ satisfies Assumptions \ref{ass} and  \ref{ass_b}. So, if $\lambda$, $\lambda^{-1}$ and $g$ are explicitly known, then 
 $X_T$ can be approximated by $\lambda^{-1}(Y_T^{(\pi_n)})$ and the error bounds \eqref{eb-intro-1} and \eqref{eb-intro-2} carry over.
\end{remark}

\section{Reduction to a quadrature problem for irregular functions of Brownian motion}\label{sec:red_quad}

In this section we will relate the analysis of the pointwise $L^2$-error of the EM scheme to a quadrature problem which will be  simpler to analyse.

In the whole paper we will denote the expectation w.r.t.~$\P$ by $\E$, the expectation w.r.t.~any other measure $\Q$ by $\E_\Q$, and the Lipschitz constant of a Lipschitz continuous function $f$ by $L_f$.
For notational simplicity we will drop the superscript $(\pi_n)$,  wherever possible.

\subsection{Notation and preliminaries}

First, we introduce a transformation $\varphi$ of the state space, which allows us to deal with the irregular part $b$ of the drift coefficient of SDE \eqref{eq:sde}.
\begin{lemma}\label{smooth}  Let Assumption \ref{ass} hold. Let $\varphi: \mathbb{R} \rightarrow\mathbb{R}$ be defined by
\begin{align}\label{varphi-solves}
\varphi (x) = \int_0^x \exp\!\left(-2 \int_0^y b(z)\, dz\right) dy,  \qquad x \in \R.
\end{align}
Then
\begin{enumerate}[(i)]
\item\label{smooth-it1} the map $\varphi$ is differentiable with bounded derivative $\varphi'$, which is absolutely continuous with bounded Lebesgue density $\varphi''\colon\R\to\R$;
\item\label{smooth-itb} the map $\varphi$ is invertible with $\varphi^{-1} \in C_b^{1}(\mathbb{R})$;
\item\label{smooth-it2} the maps $\varphi' \circ \varphi^{-1}: \mathbb{R} \rightarrow\mathbb{R}$ and $ (\varphi' a)\circ \varphi^{-1}\colon  \mathbb{R} \rightarrow\mathbb{R}$ are globally Lipschitz.
\end{enumerate}
\end{lemma}
\begin{proof}
First note that $b$ is bounded. So, by construction and the fundamental theorem of Lebesgue-integral calculus we have
$$ \varphi'(x)= \exp\!\left(-2 \int_0^x b(z)\, dz\right), \qquad  \varphi''(x) =-2b(x) \varphi'(x), \qquad x \in \mathbb{R}.$$
Since by assumption $b \in L^1(\R)$, we have that
\begin{align} \label{bound_phi'} \exp (- 2\|b\|_{L^1}) \leq \varphi'(x) \leq \exp (2\| b \|_{L^1}), \qquad x \in \R, \end{align}
which shows item \eqref{smooth-it1}. 
The last equation also implies that $\varphi$ is invertible. Moreover, we have
$$ (\varphi^{-1})^{'}(y)=\frac{1}{\varphi'(\varphi^{-1})(y))},$$ so  \eqref{bound_phi'} implies that $\varphi^{-1} \in C_b^{1}(\mathbb{R})$.
This proves item \eqref{smooth-itb}.
The Lipschitz property of $\varphi' \circ \varphi^{-1}$ and  $(\varphi'a) \circ \varphi^{-1}$ follows from the boundedness of $a, \varphi'$ and the Lipschitz property of $\varphi', \varphi^{-1}$, and $a$.
This proves item \eqref{smooth-it2}.
\end{proof}
The previous lemma implies in particular that $\varphi$ is twice differentiable almost everywhere and solves
\begin{align}
b(x)\varphi'(x) + \frac{1}{2}  \varphi''(x) = 0 \label{ode_phi} \quad  \textrm{for almost all} \,\,  \, \, x \in \R.
\end{align}

A similar transformation was introduced by Zvonkin in \cite{zvonkin1974} and the use of such techniques for the numerical analysis of SDEs goes back until \cite{talay}.\\

Now, define the transformed process $Y=(Y_t)_{t\in[0,T]}$ as $Y_t=\varphi(X_t)$. By \Ito's formula, $\mu=a+b$, and \eqref{ode_phi} we have
\begin{align*}
	Y_t
	= \varphi(\xi) + \int_0^t \varphi'(X_s)a(X_s) \,ds+ \int_0^t \varphi'(X_s) \,dW_s, \qquad t \in [0,T].
\end{align*}
Moreover, define the transformed EM scheme   $y=(y_t)_{t\in[0,T]}$ as $y_t=\varphi(x_t)$.   \Ito's formula,  \eqref{euler}, $\mu=a+b$, and \eqref{ode_phi} give
\begin{align*}
	y_t
	& = \varphi(\xi)
	+\int_0^t \left( \varphi'(x_s) (a+b)(x_{\es})+\frac{1}{2}\varphi''(x_s) \right) ds
	+ \int_0^t \varphi'(x_s)dW_s \\ & = \varphi(\xi) 
+\int_0^t \varphi'(x_s) \left(  (a+b)(x_{\es})  - (a+b)(x_s) \right) ds  
\\ & \qquad \quad \, 
	+ \int_0^t \varphi'(x_s) a(x_s) ds + \int_0^t \varphi'(x_s)dW_s,  \qquad t \in [0,T].
\end{align*}

Next, we will exploit  Girsanov's theorem, see, e.g., \cite[Section 3.5]{karatzas1991}. More precisely, we  will use a change of measure such that under the new measure $\mathbb{Q}$  the drift of the Euler scheme is removed. So let $L_T^{(\pi_n)}=\frac{d\Q}{d\P}$ be the corresponding Radon-Nikodym derivative for which  $x^{(\pi_n)}-\xi=(x^{(\pi_n)}_t-\xi)_{t \in [0,T]}$ is a Brownian motion under $\Q$, that is
\begin{equation}\label{rndensity}
L_T^{(\pi_n)}= \exp \left( -\int_0^T \mu(x_{\es}^{(\pi_n)}) dW_s - \frac{1}{2}\int_0^T\mu^2(x_{\es}^{(\pi_n)}) ds \right).
\end{equation}

We will require the following moment bound:

\begin{lemma} \label{girsanov} Let Assumption \ref{ass} hold.
For all $\varepsilon>0$ there exists a constant $c^{(L)}_{\mu,T,\varepsilon}>0$ such that
$$  \left( \E_\Q\!\left[\Big |L_T^{(\pi_n)} \Big|^{-\frac{1}{\varepsilon}}\right] \right)^{\varepsilon} \leq c^{(L)}_{\mu,T,\varepsilon}. $$
\end{lemma}

\begin{proof}
First, note that
\begin{align*}
&\E_\Q\!\left[|L_T^{(\pi_n)}|^{-\frac{1}{\varepsilon}}\right] =\E\!\left[ |L_T^{(\pi_n)}|^{\frac{\varepsilon-1}{\varepsilon}}\right] \\
& \quad =
 \E\left[\exp\!\left(\frac{\varepsilon-1}{\varepsilon}\left[
 -\int_0^T \mu(x_{\es}) dW_s - \frac{1}{2}\int_0^T\mu^2(x_{\es}) ds
 \right]\right)\right]
\\ & \quad =
  \E\left[\exp\!\left(\frac{1-\varepsilon}{\varepsilon}\left[
 \int_0^T \mu(x_{\es}) dW_s + \frac{1}{2}\int_0^T\mu^2(x_{\es}) ds
 \right]\right)\right]
\\ & \quad \le
   \exp\!\left(\frac{1-\varepsilon}{2\varepsilon}
  T \|\mu\|_\infty^2 
 \right)
  \E\left[
  \exp\!\left(\frac{1-\varepsilon}{\varepsilon}
 \int_0^T \mu(x_{\es}) dW_s\right)
 \right].
\end{align*} 
It\=o-integrals with bounded integrands have Gaussian tails, i.e. 
$$ \P \left(  \sup_{t \in [0,T]} \left| \int_0^t \mu(x_{\es}) dW_s \right| \geq \delta \right) \leq 2 \exp \left( - \frac{\delta^2}{ 4T\|\mu\|_{\infty}^2}\right ), \qquad \delta >0,$$
which is obtained by using \cite[(A.5) in Appendix A.2]{nualart2006} with $\rho=2 T \| \mu \|^2_{\infty}$.
 Since positive random variables $Z$ satisfy
\begin{align*}
 \E [Z]  = \int_0^{\infty} \P(Z \geq z) \,dz,
\end{align*}
 it follows that
\begin{align*}  
&
  \E\left[
  \exp\!\left(\frac{1-\varepsilon}{\varepsilon}
 \int_0^T \mu(x_{\es}) dW_s\right)
 \right] 
 \\ & \quad \le 
 \E\left[\exp\!\left(\left|\frac{1-\varepsilon}{\varepsilon}\right|\, \left|
 \int_0^T \mu(x_{\es}) dW_s\right| \right)
 \right] 
\\&\quad=
  \int_0^\infty \P\!\left( \exp\!\left(\left|\frac{1-\varepsilon}{\varepsilon}\right| \,\left|
 \int_0^T \mu(x_{\es}) dW_s\right|\right) \ge z \right) dz
\\&\quad \leq 1+
  \int_1^\infty \P\!\left( \left| 
 \int_0^T\mu(x_{\es}) dW_s\right| \ge \log(z)\left|\frac{\varepsilon}{1-\varepsilon}\right| \right) dz
 \\&\quad \le
1+ 2  \int_0^{\infty} \exp \left( - \frac{(\log(z))^2 \varepsilon^2}{(1-\varepsilon)^2  4T\|\mu\|_\infty^2} \right) d z 
\\ & \quad =
1+ 2  \int_{-\infty}^{\infty} \exp \left( \delta -  \frac{\delta^2}{2} \frac{\varepsilon^2}{(1-\varepsilon)^2 {2T\|\mu\|_\infty^2}} \right) d \delta
\\&\quad=
1+ \frac{4\sqrt{T\pi} (1-\varepsilon) \|\mu\|_\infty}{\varepsilon}\exp\!\left(\frac{(1-\varepsilon)^2 T \|\mu\|_\infty^2}{\varepsilon^2}\right)
< \infty,
\end{align*}
where the last step follows, e.g., from  the moment generating function for a centred Gaussian variable with variance $\frac{(1-\varepsilon)^2 2T \|\mu\|_\infty^2}{\varepsilon^2}$.
\end{proof}

Finally, we establish a technical, but straightforward estimate of weighted sums of iterated (\Ito)-integrals. 

\begin{lemma}\label{step-decomp}
Let $\psi_1,\psi_2\colon\R\to\R$ be bounded and measurable functions.
Then for all $t \in [0,T]$ we have
\begin{align*}
\E \!\left[ \left| \int_0^t \psi_1 (x_{\es}^{(\pi_n)})   \left( \int_{\es}^s \psi_2 (x_u^{(\pi_n)}) dW_u  \right) \,ds\right|^2\right]
\le \frac{t}{2}   \|\psi_1\|_\infty^2 \|\psi_2\|_\infty^2 \cdot \|\pi_n\|^2.
\end{align*} 
\end{lemma}

\begin{proof}
Since $ \psi_1(x_{\underline{\tau}})$ is $\mathcal{F}_{\underline{\tau}}$-measurable for $\tau \in [0,T]$  we have
\begin{align*}
 & \E \left[ \psi_1(x_{\es})\int_{\es}^s  \psi_2(x_u) d W_u  \, \, \psi_1(x_\et) \int_{\et}^t \psi_2(x_v) dW_v \right] 
\\ &  \quad = \E \left [ \int_{\es}^s   \psi_1(x_{\es})  \psi_2(x_u) d W_u  \, \,\int_{\et}^t  \psi_1(x_\et)  \psi_2(x_v) dW_v \right ], \qquad s,t \in [0,T].
\end{align*} Assume now  that $\et \geq s $. 
Conditioning on $\mathcal{F}_{\et }$  yields
 that
\begin{align*}
 & \E \left[ \psi_1(x_{\es})\int_{\es}^s  \psi_2(x_u) d W_u  \, \, \psi_1(x_\et) \int_{\et}^t \psi_2(x_v) dW_v \right] 
\\ &  \quad = \E \left[  \int_{\es}^s   \psi_1(x_{\es})  \psi_2(x_u) d W_u \,\, \E \left[  \, \,\int_{\et}^t  \psi_1(x_\et)  \psi_2(x_v) dW_v  \, \Big{|} \,  \mathcal{F}_{\et} \right] \right ] =0,
\end{align*}
since $ \int_{\es}^s   \psi_1(x_{\es})  \psi_2(x_u) d W_u$ is  $\mathcal{F}_{\et }$-measurable and 
$$ \E \left[  \, \,\int_{\et}^t  \psi_1(x_\et)  \psi_2(x_v) dW_v \, \Big{|} \,  \mathcal{F}_{\et} \right] =0.$$
Let $\ell \in\{0,1,\dots,n\}$ and
assume w.l.o.g.~that $\et=t_\ell$. We have
\begin{align*} 
&\E \!\left[ \left| \int_0^{t} \psi_1 (x_{\es})   \left( \int_{\es}^s \psi_2 (x_u) dW_u  \right) \,ds\right|^2\right] \\
& \quad = \E \!\left[ \left| \sum_{k=0}^{\ell-1}\int_{t_k}^{t_{k+1}} \psi_1(x_{t_k})   \int_{t_k}^s \psi_2(x_u) dW_u \,ds       + \int_{t_{\ell}}^{t} \psi_1(x_{t_{\ell}})   \int_{t_{\ell}}^s \psi_2(x_u) dW_u \,ds   \right|^2\right]
\\&\quad
=\sum_{k=0}^{\ell-1} \sum_{m=0}^{\ell-1} \E \!\left[ \left(\int_{t_k}^{t_{k+1}} \psi_1(x_{t_k})   \int_{t_k}^s \psi_2(x_u) dW_u \,ds\right)
\left( \int_{t_m}^{t_{m+1}} \psi_1(x_{t_m})   \int_{t_m}^r \psi_2(x_v) dW_v \,dr \right)\right]
\\&\qquad+ 2
\sum_{k=0}^{\ell-1} \E \!\left[ \left(\int_{t_k}^{t_{k+1}} \psi_1(x_{t_k})   \int_{t_k}^s \psi_2(x_u) dW_u \,ds\right)
\left( \int_{t_\ell}^t \psi_1(x_{t_\ell})   \int_{t_\ell}^r \psi_2(x_v) dW_v \,dr \right)\right]
\\&\qquad+
 \E \!\left[ \left| \int_{t_{\ell}}^{t} \psi_1(x_{t_{\ell}})   \int_{t_{\ell}}^s \psi_2(x_u) dW_u \,ds\right|^2\right]
\\&\quad
=\sum_{k=0}^{\ell -1}\E \!\left[ \left|\int_{t_k}^{t_{k+1}}    \int_{t_k}^s \psi_1(x_{t_k}) \psi_2(x_u) dW_u \,ds\right|^2\right]
+
 \E \!\left[ \left| \int_{t_{\ell}}^{t}   \int_{t_{\ell}}^s   \psi_1(x_{t_{\ell}}) \psi_2(x_u) dW_u \,ds\right|^2\right].
\end{align*} 
Applying the Cauchy-Schwarz inequality and using  the \Ito-isometry and  the boundedness of $\psi_1,\psi_2$ yields
\begin{equation}  \label{E3-help3} 
\begin{aligned}
&\sum_{k=0}^{\ell-1}\E \!\left[ \left|\int_{t_k}^{t_{k+1}}    \int_{t_k}^s \psi_1(x_{t_k})  \psi_2(x_u) dW_u \,ds\right|^2\right]
+
 \E \!\left[ \left| \int_{t_{\ell}}^{t}   \int_{t_{\ell}}^s \psi_1(x_{t_{\ell}})  \psi_2(x_u) dW_u \,ds\right|^2\right]
\\&\le
\|\pi_n\| \cdot \left(\sum_{k=0}^{\ell-1} \int_{t_k}^{t_{k+1}} \E \!\left[ \left|    \int_{t_k}^s \psi_1(x_{t_k})  \psi_2(x_u) dW_u \right|^2\right] ds
+
\int_{t_\ell}^{t} \E \!\left[ \left|    \int_{t_\ell}^s \psi_1(x_{t_\ell})  \psi_2(x_u) dW_u \right|^2\right] ds
\right)
\\&\le
\|\pi_n\| \|\psi_1\|_\infty^2 \|\psi_2\|_\infty^2 \cdot \left(\sum_{k=0}^{\ell-1} \int_{t_k}^{t_{k+1}} (s-t_k) ds
+\int_{t_\ell}^{t} (s-t_\ell) ds
\right)
\\ &=
 \frac{1}{2} \|\pi_n\| \|\psi_1\|_\infty^2 \|\psi_2\|_\infty^2 \cdot \left(\sum_{k=0}^{\ell-1} (t_{k+1}-t_k)^2
 +(t-t_\ell)^2
 \right)
\\&\leq 
 \frac{1}{2} \|\pi_n\|^{2} \|\psi_1\|_\infty^2 \|\psi_2\|_\infty^2 \cdot \left(\sum_{k=0}^{\ell-1} (t_{k+1}-t_k)
 +(t-t_\ell)
 \right)
 \\&=
\frac{t}{2} \|\pi_n\|^2  \|\psi_1\|_\infty^2 \|\psi_2\|_\infty^2
.
\end{aligned}
\end{equation}

\end{proof}

\subsection{Reduction to a quadrature problem}

Now we relate the error of the EM scheme $x^{(\pi_n)}=(x^{(\pi_n)})_{t \in [0,T]}$ to the error of a weighted quadrature problem.

\begin{theorem}\label{main_1}  Let Assumption \ref{ass} hold. Then, for all $\varepsilon\in(0,1)$ there exists a constant $C^{(R)}_{\varepsilon, a,b,T}>0$ such that 
$$  \sup_{t \in [0,T]} \E\!\left[|X_t-x_t^{(\pi_n)}|^2\right] \leq  C^{(R)}_{\varepsilon, a,b,T} \cdot  \left(\|\pi_n\|^2 +  \sup_{t \in [0,T]} |\mathcal{W}_t^{(\pi_n)}|^{1-\varepsilon} \right),$$
where
\begin{align*}
\mathcal{W}^{(\pi_n)}_t &= \E\!\left[ \left| \int_0^t   \varphi'(W_s+\xi)  \left(b(W_s+\xi)-  b(W_\es+\xi)\right) ds \right|^2\right], \qquad t \in [0,T].
\end{align*}
\end{theorem}
\medskip

\begin{proof}
\textit{ Step 1.}
First note that by Lemma \ref{smooth} we have
\begin{align}\label{trafoappl}
  \E\left[|X_t-x_t|^2 \right] = \E\left[\left|\varphi^{-1}(Y_t)-\varphi^{-1}(y_t)\right|^2\right] 
   \le L_{\varphi^{-1}}^2  \E\left[|Y_t-y_t|^2 \right], \qquad t \in [0,T].
   \end{align}
Furthermore, we have for all $t \in [0,T]$ that
\begin{equation} \label{eq:Y-y}
\begin{aligned}
Y_t-y_t  = E_t &+ \int_0^t \left( (\varphi'a)(\varphi^{-1}(Y_s))-   (\varphi'a)( \varphi^{-1}(y_s) )\right) ds \\ & + \int_0^t \left( \varphi' ( \varphi^{-1}(Y_s))-  \varphi' ( \varphi^{-1}(y_s)) \right) dW_s,
\end{aligned}
\end{equation}
where
$$ E_t= \int_0^t  \varphi'(x_s) \left((a+b)(x_s)- (a+b)(x_{\es}) \right) ds.$$
Applying  the representation \eqref{eq:Y-y}, the Cauchy-Schwarz inequality, the \Ito-isometry,  and  Lemma \ref{smooth}  we obtain  for all $t\in [0,T]$
that
\begin{align*}
  \E\left[|Y_t-y_t|^2 \right] 
&\leq  3 \E\!\left[|E_t|^2 \right]   +   3  \E \left[ \left|
\int_0^t \left( (\varphi'a)(\varphi^{-1}(Y_s))-  (\varphi'a)(\varphi^{-1}(y_s)) \right) ds \right|^2 \right] 
	 \\ & \qquad \qquad \quad \,\, + 3  \E \left[ \left| \int_0^t \left( \varphi'(\varphi^{-1}(Y_s))-  \varphi'(\varphi^{-1}(y_s)) \right) dW_s\right|^2 \right] 
\\&\leq  3     \E\left[|E_t|^2 \right] 
+   3  \left(T L_{(\varphi'a) \circ \varphi^{-1}}^2 +  L_{\varphi' \circ \varphi^{-1}}^2 \right)
\int_0^t \E \left[ \left| Y_s -y _s \right|^2 \right ] ds  \\ & \leq   3  \sup_{u \in [0,t]} \E\left[|E_u|^2 \right] 
+   3  \left(T L_{(\varphi'a) \circ \varphi^{-1}}^2 +  L_{\varphi' \circ \varphi^{-1}}^2 \right)
\int_0^t   \sup_{u \in [0,s]} \E \left[ \left| Y_u-y _u \right|^2 \right ] ds  
.
\end{align*} 
This estimate, Gronwall's lemma, and \eqref{trafoappl} establish that there exists a constant $c^{(1)}_{a,b,T}>0$ such that 
\begin{align}\label{est1}
    \sup_{t \in [0,T]}\E\left[|X_t-x_t|^2 \right] 
\le c^{(1)}_{a,b,T} \,      \sup_{t \in [0,T]} \E\left[|E_t|^2 \right].  
\end{align}
Clearly, we have that
\begin{align}\label{SumE}  \E\left[|E_{t}|^2 \right] =
\E\left[\left|\int_0^{t}  \varphi'(x_s) \left((a+b)(x_s)- (a+b)(x_{\es}) \right) ds\right|^2 \right]
 \leq 3(\mathcal{E}_1(t)+\mathcal{E}_2(t)+\mathcal{E}_3(t)),
 \end{align}
where
\begin{align*}
 \mathcal{E}_1(t) &= \E\left[\left|\int_0^{t}   \varphi'(x_{\es})  \left( a(x_s)-  a(x_{\es}) \right) ds\right|^2 \right], \\
 \mathcal{E}_2(t) &= \E\left[\left|\int_0^{t}  \left( \varphi'(x_s) -  \varphi'(x_{\es})  \right) \left(  a(x_s)- a(x_{\es}) \right) ds\right|^2 \right],\\
  \mathcal{E}_3(t) &= \E\left[\left|\int_0^{t}\varphi'(x_s)  \left(b(x_s)-  b(x_{\es})\right) ds\right|^2 \right].
\end{align*}

We will first deal with $\mathcal{E}_1$ and $\mathcal{E}_2$ using standard tools, then we will rewrite 
$\mathcal{E}_3$ using a Girsanov transform.\\

\textit{Step 2.} For estimating $\mathcal{E}_1$ and $\mathcal{E}_2$ note that for all $s \in [0,T]$ we have
\begin{equation} \label{fourthmoment}
\begin{aligned}
  \E\!\left[|x_s-x_{\es}|^4 \right] & =  \E\!\left[\left| \int_\es^s(a+b)(x_\et) dt+(W_s-W_\es)  \right|^4 \right]
  \\&\quad\le8 \, \E\!\left[\left| \int_\es^s |(a+b)(x_\et)| dt\right|^4\right]+8\,\E\!\left[|W_s-W_\es |^4 \right]
    \\&\quad\le8(s-\es)^4  \E \left[ \sup_{t \in [0,T]}|(a+b)(x_\et)|^4 \right]+24 (s-\es)^2
    \\&\quad\le8 \|a+b\|_\infty^4 \|\pi_n\|^4+24 \|\pi_n\|^2.
\end{aligned}
\end{equation} 
To estimate $\mathcal{E}_2$ we apply the Cauchy-Schwarz inequality and   $ 2 xy \leq x^2 +y^2$   to obtain that
\begin{align*}
\mathcal{E}_2{(t)}&\leq {t}  \int_0^{t}\E\!\left[\left| \varphi'(x_s) -  \varphi'(x_{\es})  \right|^2 \left|  a(x_s)- a(x_{\es})  \right|^2\right] ds\\
&\le \frac{{t}}{2}  \int_0^{t} \E\!\left[\left|  \varphi'(x_s) -  \varphi'(x_{\es})  \right|^4 + \left| a(x_s)- a(x_{\es}) \right|^4\right] ds.
\end{align*}
Since $\varphi'$ and $a$ are globally Lipschitz, \eqref{fourthmoment} yields
\begin{equation}
\begin{aligned} 
 \label{E4}
\mathcal{E}_2{(t)} & \leq \frac{t}{2} \left( L_{\varphi'}^4+ L_{a}^4 \right) \int_0^{t} \E\!\left[|x_s-x_{\es}|^4 \right]ds  \\ & \leq 
 \frac{t^2}{2} \left( L_{\varphi'}^4+ L_{a}^4 \right) \left( 8 \|a+b\|_\infty^4 \|\pi_n\|^4+24 \|\pi_n\|^2 \right).
\end{aligned} 
\end{equation}
Recall that $a \in C^2_b(\mathbb{R})$. So, It\=o's formula  yields
\begin{align*}
&\int_0^{t}  \varphi'(x_{\es})  \left( a(x_s)-  a(x_{\es}) \right) ds \\ &  \quad= \int_0^{t} \varphi'(x_{\es}) \left( 
\int_{\es}^s \left(a'(x_u) (a+b)(x_{\es}) + \frac{1}{2}a''(x_u)\right) du + \int_{\es}^s a'(x_u) d W_u \right)ds.
\end{align*}
Hence, we have
\begin{equation}\label{E3-help1}
\begin{aligned}
\mathcal{E}_1{(t)}  & \leq 2\, \E \!\left[\left|   \int_0^{t} \varphi'(x_{\es}) 
\int_{\es}^s \left(a'(x_u) (a+b)(x_{\es}) +\frac{1}{2} a''(x_u) \right) du \, ds\right|^2 \right] 
\\ & \quad+
 2\, \E \!\left[ \left| \int_0^{t}  \varphi'(x_\es)   \int_{\es}^s a'(x_u) dW_u \,ds\right|^2\right] 
\\& \leq 
 2\, \E \!\left[\left|   \int_0^{t}  \int_{\es}^s  \left| \varphi'(x_{\es})\right| 
 \left| a'(x_u) (a+b)(x_{\es}) +\frac{1}{2} a''(x_u) \right|   du \, ds\right|^2 \right] 
\\ & \quad+
 2\, \E \!\left[ \left| \int_0^{t}  \varphi'(x_\es)   \int_{\es}^s a'(x_u) dW_u \,ds\right|^2\right] .
\end{aligned} 
\end{equation}
Using that
  $a,b,a',a'',\varphi'$ are bounded,  gives 
$$ \sup_{u,s \in[0,T]}  \left| \varphi'(x_{\es})\right| 
 \left| a'(x_u) (a+b)(x_{\es}) +\frac{1}{2} a''(x_u) \right|  \leq 
  \|\varphi'\|_\infty  \left( \|a'\|_\infty \, \|a+b\|_\infty + \frac{1}{2}\|a''\|_\infty \right).
 $$  So we obtain 
\begin{equation}\label{E3-help2}
\begin{aligned}
& \E \!\left[\left|   \int_0^{t} \int_{\es}^s \varphi'(x_{\es}) 
 \left(a'(x_u) (a+b)(x_{\es}) +\frac{1}{2} a''(x_u)\right) du \, ds\right|^2 \right] \\ & \qquad  \leq {t}^2 \|\varphi'\|_\infty^2  \left( \|a'\|_\infty \, \|a+b\|_\infty + \frac{1}{2}\|a''\|_\infty \right)^2  \|\pi_n\|^2.
\end{aligned} 
\end{equation}
Combining \eqref{E3-help1} with \eqref{E3-help2} and applying Lemma \ref{step-decomp} to the second summand of   \eqref{E3-help1} yield
\begin{align}\label{E3}
 \mathcal{E}_1(t)\le 2{t}^2 \|\varphi'\|_\infty^2  \left( \|a'\|_\infty \, \|a+b\|_\infty + \frac{1}{2}\|a''\|_\infty \right)^2  \|\pi_n\|^2
 + {t} \|a'\|_\infty^2 \|\varphi'\|_\infty^2  \|\pi_n\|^2.
\end{align}
Thus, \eqref{E3} and \eqref{E4} imply that there exists a constant $c^{(2)}_{a,b,T}>0$ such that
\begin{align}\label{est2}
\mathcal{E}_1(t) +\mathcal{E}_2(t)
\le c^{(2)}_{a,b,T} \, \|\pi_n\|^2
\end{align}
for all $t \in [0,T]$.
So, combining \eqref{est1}, \eqref{SumE}, and \eqref{est2}, we obtain that there exists a constant $c_{ a,b,T}^{(3)} >0$ such that
\begin{align}\label{mainest1} \sup_{t \in [0,T]}
\E\!\left[|X_t-x_t|^2 \right]\leq   c_{ a,b,T}^{(3)} \left( \|\pi_n\|^2 + \sup_{t \in [0,T]} \mathcal{E}_3(t) \right).
\end{align}

\textit{Step 3:}  Now we use the Girsanov-transform with density $L_T=L_T^{(\pi_n)}$ as in \eqref{rndensity}, i.e.~as before we change the measure to $\mathbb{Q}$ to replace the Euler scheme $x=x^{(\pi_n)}$ by $W+\xi$.
For $\varepsilon \in (0,1)$, H\"older's inequality and Lemma \ref{girsanov} yield 
\begin{align*}
 \mathcal{E}_3(t)  &= \E_\Q\left[L_T^{-1} \left|\int_0^T \mathbf{1}_{[0,t]}(s) \varphi'(x_s)  \left(b(x_s)-  b(x_{\es})\right)ds\right|^2 \right]
\\&\le  \left( \E_\Q\!\left[|L_T|^{-\frac{1}{\varepsilon}}\right]\right) ^{\!\varepsilon}   \left( \E_\Q \!\left[ \left| \int_0^t   \varphi'(x_s)  \left(b(x_s)-  b(x_{\es})\right)ds \right|^{\frac{2}{1-\varepsilon}} \right]\right)^{\!1-\varepsilon}
\\ & \leq c^{(L)}_{\mu,T,\varepsilon}  \left( \E \!\left[ \left| \int_0^t  \varphi'(W_s+\xi)  \left(b(W_s+\xi)-  b(W_\es+\xi)\right)ds\right|^{\frac{2}{1-\varepsilon}} \right]\right)^{\!1-\varepsilon}.
\end{align*} Note that  $c^{(L)}_{\mu,T,\varepsilon}$ is independent of $\pi_n$.
Since 
\begin{align*}
 & \left| \int_0^{t}   \varphi'(W_s+\xi)  \left(b(W_s+\xi)-  b(W_\es+\xi)\right) ds \right|^{\frac{2}{1-\varepsilon}} \\& 
\,\, =     \left| \int_0^{t}       \varphi'(W_s+\xi)  \left(b(W_s+\xi)-  b(W_\es+\xi)\right)ds \right|^{\frac{2\varepsilon}{1-\varepsilon}}   \left| \int_0^{t}       \varphi'(W_s+\xi)  \left(b(W_s+\xi)-  b(W_\es+\xi)\right)ds \right|^2 
\\
& \,\,  \leq   (2 {t}    \| \varphi' b\|_{\infty})^{\frac{2\varepsilon}{1-\varepsilon}} \left| \int_0^{t}       \varphi'(W_s+\xi)  \left(b(W_s+\xi)-  b(W_\es+\xi)\right)ds \right|^2,
\end{align*}
we obtain for all $t \in [0,T]$,
\begin{align}\label{mainest2}   \mathcal{E}_3{(t)} \leq  c^{(L)}_{\mu,T,\varepsilon}  (2t\| \varphi' b\|_{\infty})^{2\varepsilon} (\mathcal{W}^{(\pi_n)}_{t})^{1-\varepsilon}. \end{align}

Combining \eqref{mainest1} and \eqref{mainest2} proves the theorem.

\end{proof}

\begin{remark}\label{quad_rem}
The term $\mathcal{W}^{(\pi_n)}_{t}$ corresponds to the mean-square error of a  weighted quadrature problem, namely the prediction of
 $$ I= \int_0^T \mathcal{Y}_s Z_s  ds$$ 
by the quadrature rule $${I}^{(\pi_n)}= \sum_{k=0}^{n-1}  Z_{t_k} \int_{t_k}^{t_{k+1}} \mathcal{Y}_s ds, $$ where $$\mathcal{Y}_t= \varphi'(W_t + \xi)=  \exp\left(-2\int_0^{W_t+\xi} b(z)dz \right) , \qquad t \in [0,T],$$ is a  random weight function, 
and  the process $Z$ given by $$Z_t=b(W_t+ \xi), \qquad t \in [0,T],$$ is  evaluated at $t_0, \ldots, t_{n-1}$.
 Related unweighted integration problems, i.e.~with $\mathcal{Y}=1$ and  $Z$  given by irregular functions of  stochastic processes   such as (fractional) Brownian motion, SDE solutions, or  general Markov processes, have recently been studied in
\cite{ngo2011,Kohatsu-higa-etal,altmeyer2017}. 
In particular, Sobolev-Slobodeckij spaces have been used in this context by \citet{altmeyer2017}.

 The study  of  quadrature problems for stochastic processes goes back to the seminal works of \citet{sy1,sy2,sacks_3,sy3}.
 
Note also that the approximation of It\=o-integrals of the form 
$\int_0^1 g(s)dW_s$, where $g$ has fractional Sobolev regularity of order $\kappa \in (0,1)$ by means of a Riemann-Maruyama approximation based on a randomly shifted grid has been studied in \cite{raphael}.
\end{remark}

\section{Analysis of the quadrature problem}

For the analysis of 
\begin{align*}
\mathcal{W}^{(\pi_n)}_{t} &= \E\!\left[ \left| \int_0^t   \varphi'(W_s+\xi)  \left(b(W_s+\xi)-  b(W_\es+\xi)\right) ds \right|^2\right], \qquad t \in [0,T],
\end{align*}
we assume additionally Assumption \ref{ass_b}, i.e.~that the irregular part of the drift has Sobolev-Slobodeckij regularity of order $\kappa\in(0,1)$.

\medskip

\subsection{Analytic preliminaries}

As a preparation we need:

\begin{lemma} \label{sobolev}
  Let Assumptions \ref{ass} and \ref{ass_b}   hold. Then we have
$  |\varphi' b|_{\kappa}< \infty$.
\end{lemma}

\begin{proof} 
We  can write
\begin{align*}
 (\varphi'b)(x)- (\varphi'b)(y) &= \varphi'(x)(b(x)-b(y))  + b(y) (\varphi'(x)-\varphi'(y)).
\end{align*}
Since $\varphi'$ is bounded,  we have that
$$
  \int_{\mathbb{R}} \int_{\mathbb{R}} \frac{ |\varphi'(x )(b(x )-b(y))|^2  }{|x-y|^{1+ 2 \kappa}} dx dy   
\leq 
 \|\varphi'\|_\infty^2 |b|_{\kappa}^2.
$$
Moreover, the boundedness of $\varphi''$  implies
$$|  b(y) (\varphi'(x)-\varphi'(y))|^2  \leq  |b(y)|^2 \|\varphi''\|_{\infty}^2 |x-y|^2. $$ 
Since $b$ is bounded and $b \in L^1(\mathbb{R})$, it follows that $b \in L^2(\mathbb{R})$.
Hence, for all $ \kappa \in (0,1)$ we have
\begin{align*}
 \int_{\mathbb{R}} \int_{y-1}^{y+1} \frac{|b(y) (\varphi'(x)-\varphi'(y))|^2}{|x-y|^{1+2\kappa}} dx dy 
&\le 2 \|\varphi''\|_{\infty}^2  \int_{\mathbb{R}}  |b(y)|^2 \int_{y}^{y+1} |x-y|^{1-2\kappa} dx dy \\
& = \frac{1}{1-\kappa}  \|\varphi''\|_{\infty}^2  \| b \|_{L^2}^2 < \infty.
\end{align*}
Furthermore, the boundedness of $\varphi'$ yields
\begin{align*}
 \int_{\mathbb{R}} \int_{y+1}^{\infty} \frac{|b(y) (\varphi'(x)-\varphi'(y))|^2}{|x-y|^{1+2\kappa}} dx dy 
&\leq  {4} \|\varphi'\|_{\infty}^2  \int_{\mathbb{R}}  |b(y)|^2 \int_{y+1}^{\infty}  |x-y|^{-1-2\kappa} dx dy \\
& =   \frac{2}{\kappa}    \|\varphi'\|_{\infty}^2   \| b \|_{L^2}^2< \infty,
\end{align*}
and analogously
\begin{align*}
 \int_{\mathbb{R}} \int^{y-1}_{-\infty} \frac{|b(y) (\varphi'(x)-\varphi'(y))|^2}{|x-y|^{1+2\kappa}} dx dy 
 \le \frac{2}{\kappa}  \|\varphi'\|_{\infty}^2   \| b \|_{L^2}^2< \infty.
\end{align*}
Thus, the assertion follows.
\end{proof}

Since the Sobolev-Slobodeckij semi-norm is shift invariant, Lemma  \ref{sobolev} also yields:

\begin{corollary} \label{sobolev-2}
  Let Assumptions \ref{ass} and \ref{ass_b}   hold. Then we have $\mathbb{P}$-a.s. that
$$ |\varphi'b(\cdot + \xi)|_{\kappa} = |\varphi'b|_{\kappa}< \infty .$$
\end{corollary}

In the following, we will frequently use that for all $p\geq 0$ there exists a constant $c_p > 0$ such that  for all $w \in\R$ we have
\begin{align}\label{exp-est}
|w|^p \exp(-w^2/2) \leq c_p \exp(-w^2/4).
\end{align}

A crucial tool will be the following bound on the Gaussian density:

\begin{lemma} \label{lem_pab} Let  $t>s>0$ and 
\begin{align}\label{pts}
p_{t,s}(x,y)= \frac{1}{2 \pi} \frac{1}{\sqrt{s(t-s)}} \exp \left(  - \frac{(x-y)^2}{2(t-s)} - \frac{y^2}{2s} \right), \qquad x,y \in \mathbb{R}.
\end{align}
Then we have
\begin{align}\label{p-deriv}
 \frac{ \partial^2}{\partial t \partial s} p_{t,s}(x,y)&= \nonumber\frac{1}{4}  p_{t,s}(x,y) \left( \frac{y^2}{s^2} - \frac{1}{s} \right) \left( \frac{(y-x)^2}{(t-s)^2} - \frac{1}{t-s} \right) \\ & \qquad  - \frac{1}{4}  p_{t,s}(x,y) \left( \frac{(y-x)^2}{(t-s)^2} - \frac{1}{t-s} \right)^2  \\ & \qquad + \frac{1}{2}  p_{t,s}(x,y)  \left( \frac{2(y-x)^2}{(t-s)^3} - \frac{1}{(t-s)^2} \right) \nonumber
\end{align}
and there exists a constant $\densityconst>0$ such that 
\begin{align} \label{est-pab} 
 -|x-y|^{1+2\kappa}  \frac{ \partial^2}{\partial t \partial s} p_{t,s}(x,y) \leq  \densityconst \left( |t-s|^{\kappa -2} s^{-1/2}+ |t-s|^{\kappa -1} s^{-3/2} \right).
\end{align}
\end{lemma}

\begin{proof}

Straightforward calculations yield the first assertion \eqref{p-deriv}.

Moreover, we have
\begin{equation}\label{help-densityp}
\begin{aligned}
 &-|x-y|^{1+2\kappa}  \frac{ \partial^2}{\partial t \partial s} p_{t,s}(x,y) \\ &\quad \leq 
\left[  \frac{3}{4} \frac{1}{(t-s)^2} + \frac{1}{4} \frac{(x-y)^4}{(t-s)^4}  + \frac{1}{4} \frac{y^2}{s^2}\frac{1}{t-s}+  \frac{1}{4}  \frac{(x-y)^2}{s(t-s)^2}\right]
|x-y|^{1+2\kappa}  p_{t,s}(x,y)
\\ &\quad =
\frac{1}{8\pi} \frac{1}{\sqrt{s(t-s)}}  \frac{1}{|t-s|^{3/2-\kappa}}  \exp\!\left(-\frac{(x-y)^2}{2(t-s)}\right)\exp\!\left(-\frac{y^2}{2s}\right) 
\\& \qquad \qquad \quad \qquad \qquad \qquad    \qquad \qquad   \times
\left[   3 \frac{|x-y|^{1+2\kappa} }{|t-s|^{1/2+\kappa}} +
 \frac{|x-y|^{5+2\kappa} }{|t-s|^{5/2+\kappa}}\right]   \\ & \qquad +
\frac{1}{8\pi} \frac{1}{\sqrt{s(t-s)}}  \frac{1}{s|t-s|^{1/2-\kappa}}  \exp\!\left(-\frac{(x-y)^2}{2(t-s)}\right)\exp\!\left(-\frac{y^2}{2s}\right) 
\\& \qquad \qquad \quad \qquad \qquad \qquad    \qquad \qquad   \times
\left[   
 \frac{y^2}{s} \frac{|x-y|^{1+2\kappa} }{|t-s|^{1/2+\kappa}} + \frac{|x-y|^{3+2\kappa} }{|t-s|^{3/2+\kappa}}\right].
\end{aligned}
\end{equation}
Setting $w^2=(x-y)^2/(t-s)$ respectively $w^2=y^2/s$ in \eqref{exp-est}, we obtain that for every $p \geq 0$ there exists a constant $c_{2p}>0$ such that for all $t>s>0$ and $x,y \in \mathbb{R}$ it holds
\begin{align*}
\frac{|x-y|^{2p}}{|t-s|^{p}}  \exp\!\left(- \frac{|x-y|^{2}}{2(t-s)} \right)  & \leq c_{2p}  \exp\! \left(- \frac{|x-y|^{2}}{4(t-s)} \right) , \\ 
\frac{y^{2p}}{s^p}   \exp\!\left(- \frac{y^{2}}{2s} \right) &   \leq c_{2p}   \exp\! \left(- \frac{y^{2}}{4s} \right) .
\end{align*} 
This and \eqref{help-densityp} establish that there exist constants $c_{1+2\kappa},c_{5+2\kappa},c_{2},c_{3+2\kappa}>0$ such that
\begin{align*}
 &-8 \pi \sqrt{s(t-s)}|x-y|^{1+2\kappa}  \frac{ \partial^2}{\partial t \partial s} p_{t,s}(x,y) \\ &\quad \leq 
\left[   \frac{3 c_{1+2\kappa}}{|t-s|^{3/2-\kappa}} \exp\!\left(-\frac{(x-y)^2}{4(t-s)}-\frac{y^2}{2s}\right)
+  \frac{c_{5+2\kappa}}{|t-s|^{3/2-\kappa}}   \exp\!\left(-\frac{(x-y)^2}{4(t-s)}-\frac{y^2}{2s}\right)
\right.\\&\qquad+\left.
 \frac{c_2 c_{1+2\kappa}}{s|t-s|^{1/2-\kappa}} \exp\!\left(-\frac{(x-y)^2}{4(t-s)}-\frac{y^2}{4s}\right)
 +\frac{c_{3+2\kappa}}{s|t-s|^{1/2-\kappa}}  \exp\!\left(-\frac{(x-y)^2}{4(t-s)}-\frac{y^2}{2s}\right) \right].
\end{align*}
Hence, there exists a constant $\densityconst>0$ such that
\begin{align*}
-|x-y|^{1+2\kappa}  \frac{ \partial^2}{\partial t \partial s} p_{t,s}(x,y) 
&\leq \densityconst  |t-s|^{\kappa -3/2} \cdot  \frac{1}{\sqrt{s(t-s)}} \exp \left(  - \frac{(x-y)^2}{4(t-s)} - \frac{y^2}{2s} \right) \\ & \quad + \densityconst  |t-s|^{\kappa -1/2} s^{-1} \cdot  \frac{1}{\sqrt{s(t-s)}} \exp \left(  - \frac{(x-y)^2}{4(t-s)} - \frac{y^2}{4s} \right).
\end{align*}
Using that the exponential terms above are bounded by one, we have
\begin{align*}
-|x-y|^{1+2\kappa}  \frac{ \partial^2}{\partial t \partial s} p_{t,s}(x,y) 
&\leq \densityconst  |t-s|^{\kappa -2} s^{-1/2} + \densityconst  |t-s|^{\kappa -1} s^{-3/2}.
\end{align*}
\end{proof}

\subsection{Stochastic preliminaries}

We denote by $\phi_{\vartheta}$ the function $\phi_{\vartheta}(x)=\frac{1}{\sqrt{2 \pi \vartheta}}\exp\big(-\frac{x^2}{2\vartheta}\big)$, $x\in \mathbb{R}$, $\vartheta>0$.
We require the following auxiliary result.

\begin{lemma} \label{lem_1} 
Let $\kappa \in (0,1)$, and let $f: \mathbb{R} \rightarrow \mathbb{R}$ be measurable  such that $|f|_{\kappa}< \infty.$
Then there exists a constant $c_{\kappa}>0$ such that for all $0 < s \leq t \leq T$ 
we have
$$  \E \!\left[  | f(W_t +\xi) - f(W_{s}+\xi)|^2  \right] \leq c_{\kappa} |f|_{\kappa}^2 \cdot (t-s)^{\kappa} s^{-1/2}. $$
\end{lemma}

\begin{proof}
Clearly, we have
$$ \E \!\left[ |  f(W_t+ \xi) - f(W_{s}+\xi)|^2 \right] = \E \left[  \E \!\left[ |  f(W_t+ \xi) - f(W_{s}+\xi)|^2  \big{|} \mathcal{F}_0 \right] \right]. $$
Since $W$ is independent of $\mathcal{F}_0$, we obtain
\begin{align*}
&   \E \!\left[ |  f(W_t+ \xi) - f(W_{s}+\xi)|^2 \big{|} \mathcal{F}_0 \right] = \int_{\R} \int_{\R}  (f(x+y+\xi)-f(y+\xi))^2  \phi_{t-s}(x) \phi_{s}(y) dy dx.
\end{align*}
Now write
\begin{align*}
 &  \int_{\R} \int_{\R} (f(x+y+ \xi)-f(y+\xi))^2\phi_{t-s}(x) \phi_{s}(y) dy dx 
    \\ &  \quad = (t-s)^{1/2 + \kappa} \int_{\R} \int_{\R} \frac{(f(x+y+\xi)-f(y+\xi))^2}{|x|^{1+ 2 \kappa}} \frac{|x|^{1+2 \kappa}}{(t-s)^{1/2 + \kappa}} \phi_{t-s}(x) \phi_{s}(y) dy dx .
   \end{align*}
Next we use  \eqref{exp-est} with $w^2=x^2/(t-s)$.
This yields for all $x\in\R$ the estimate
\begin{align*}
\frac{|x|^{1+2 \kappa}}{(t-s)^{1/2 + \kappa}}
\phi_{t-s}(x)
&= 
\frac{|x|^{1+2 \kappa}}{(t-s)^{1/2 + \kappa}}
\frac{1}{\sqrt{2 \pi(t-s)}} \exp\! \left(- \frac{x^2}{2(t-s)} \right)
\\ &  \le 
c_{1+2\kappa} \frac{1}{\sqrt{2 \pi(t-s)}}   \exp\! \left(- \frac{x^2}{4(t-s)}\right)
 \leq 
c_{1+2\kappa} \frac{1}{\sqrt{2 \pi(t-s)}}  .
\end{align*} 
Since moreover
$ \phi_{s}(y) \leq \frac{1}{\sqrt{2 \pi s}}$, Corollary \ref{sobolev-2} yields
\begin{align*}
\E\!\left[  |  f(W_t + \xi )) - f(W_{s}+\xi)|^2 \right] 
&\leq
\frac{c_{1+2\kappa}}{2\pi } (t-s)^{\kappa}s^{-1/2} \int_{\R} \int_{\R} \E \left[  \frac{(f(z+\xi)-f(y+\xi))^2}{|z-y|^{1+2 \kappa}} \right]
dy dz 
\\&=
\frac{c_{1+2\kappa}}{2\pi } (t-s)^{\kappa}s^{-1/2} |f|_{\kappa}^2,
\end{align*}
which is the desired statement.
\end{proof}

The following Lemma deals with an  integration problem  seemingly similar to $\mathcal{W}^{(\pi_n)}$. However, the transformation of $W+\xi$ has significantly more smoothness here.

\begin{lemma}\label{step-decomp_2}
Let $\kappa \in (0,1)$ and $\psi_3,\psi_4\colon\R\to\R$ be bounded and measurable functions. Moreover, let $\psi_3$ be absolutely continuous with bounded Lebesgue density $\psi'_3\colon\R\to\R$ that satisfies $|\psi_3'|_{\kappa} < \infty$.
Then, there exists a constant $c^{(qs)}_{\psi_3, \psi_4,\kappa,T}>0$ such that 
\begin{align*}  & \sup_{t \in [0,T]} \E\!\left[
\left| \int_0^{t}     \left( \psi_3(W_{s}+\xi)  -  \psi_3(W_\es+\xi) \right)  \psi_4(W_{\es}+\xi) ds \right|^2\right]  \\ & \qquad  \qquad \qquad  \leq    c^{(qs)}_{\psi_3, \psi_4, \kappa, T}  \left( 1+ \sum_{k=1}^{n-1} t_k^{-1/2}(t_{k+1}-t_k)\right)\cdot \| \pi_n \|^{1+\kappa}. 
\end{align*}
\end{lemma}

\begin{proof}
The fundamental theorem of Lebesgue-integral calculus implies for all $t\in[0,T]$ that
\begin{equation}\label{W21W22}
\begin{aligned}
& \E\!\left[
 \left| \int_0^{t}     \left( \psi_3(W_{s}+\xi)  -  \psi_3(W_\es+\xi) \right)  \psi_4(W_{\es}+\xi) ds \right|^2\right] \\ 
& = \E \left[ \left| \int_0^{t}     \int_0^1  \left( W_s -W_{\es} \right) \psi'_3\left( \xi + W_{\es} + \gamma (W_s - W_{\es}) \right)  \psi_4(W_{\es}+\xi)  d \gamma ds \right|^2\right]
\\ & \leq 2 \left(\mathcal{E}_{1}{ (t)} +\mathcal{E}_{2}{ (t)} \right),
\end{aligned}
\end{equation}
where
\begin{align*}
\mathcal{E}_{1}{(t)}&=
 \E \left[ \left| \int_0^{t}   \int_0^1   \left[\psi'_3(W_{\es}+\xi + \gamma(W_s-W_{\es}))
 - \psi'_3(W_{\es}+\xi)
  \right]\left( W_s -W_{\es} \right)   \psi_4(W_{\es}+\xi)    d \gamma ds \right|^2\right],
  \\
\mathcal{E}_{2}{(t)}&
 = \E \left[ \left| \int_0^{{t}}    (\psi'_3 \psi_4)(W_{\es}+\xi)\left( W_s -W_{\es} \right)   ds \right|^2\right]  .
\end{align*}
For the second term, we  apply Lemma \ref{step-decomp} with $\psi_1=\psi_3'\psi_4$, $\psi_2=1$ and obtain 
\begin{equation}\label{est-W22}
\mathcal{E}_{2} {(t)} \leq
\frac{t }{2} \|\psi'_3 \psi_4\|_\infty^2 \|  \pi_n \|^2. 
\end{equation}
For $\mathcal{E}_1$ the Cauchy-Schwarz inequality gives
\begin{align*}
\mathcal{E}_{1}{ (t)} \le {t} \int_0^{ t}    \int_0^1  \E\!\left[ \left| \left[ \psi_3'(W_{\es}+\xi + \gamma(W_s-W_{\es}))
 - \psi_3'(W_{\es}+\xi)
  \right]\left( W_s -W_{\es} \right)   \psi_4 (W_{\es}+\xi)   \right|^2 \right] d \gamma ds.
\end{align*}
Splitting the time integral yields
\begin{equation}\label{est-W21-1}
\mathcal{E}_{1}{(t)}\leq 4 {t} \| \psi'_3 \|_{\infty}^2 \| \psi_4 \|_{\infty}^2 t_1^2  +
 {t }  \widetilde{\mathcal{E}}_{1}{(t)} \cdot\1_{[t_1,T]}(t)
 \le
 4 {t}\| \psi'_3 \|_{\infty}^2 \| \psi_4 \|_{\infty}^2  \|  \pi_n \|^2  +
 {t}  \widetilde{\mathcal{E}}_{1}{(t)}\cdot\1_{[t_1,T]}(t)
 ,
\end{equation}
where
$$ \widetilde{\mathcal{E}}_{1}{(t)}=
 \int_{t_1}^{t}   \int_0^1  \E\!\left[ \Big| \left[ \psi_3'(W_{\es}+\xi + \gamma(W_s-W_{\es}))
 - \psi_3'(W_{\es}+\xi)
  \right]\left( W_s -W_{\es} \right)   \psi_4 (W_{\es}+\xi)   \Big|^2\right]d\gamma ds. $$
Now write
\begin{align*}
  \widetilde{\mathcal{E}}_{1}{(t)}&=  \int_{t_1}^{t}  \int_0^1\E\!\left[  \E\! \left[ \Big| \left[ \psi_3'(W_{\es}+\xi + \gamma(W_s-W_{\es}))
 - \psi_3'(W_{\es}+\xi)
  \right]\left( W_s -W_{\es} \right)   \psi_4(W_{\es}+\xi)   \Big|^2 \Big{|} \mathcal{F}_0 \right] \right]d \gamma ds \\
& \,\, = \E\!\Big [  \int_{\mathbb{R}}  \int_{\mathbb{R}}    \int_{t_1}^{t}   \int_0^1   [\psi_3'(y+ \xi+x) 
 - \psi_3'(y+\xi)]^2 \left( \frac{x}{\gamma} \right)^2 (\psi_4( y+\xi ))^2 
   \\&\ \qquad \qquad \qquad \qquad \qquad \qquad \qquad   \qquad \qquad \qquad \qquad \qquad 
  \times \phi_{\gamma^2(s-\es)}(x) \phi_{\es}(y)   d \gamma ds   dy dx\Big].
\end{align*}
With $\phi_{ \es}(y) \leq \frac{1}{ \sqrt{2\pi \es}}$ for all $y \in \mathbb{R}$, we obtain
\begin{align*}
 \widetilde{\mathcal{E}}_{1}{(t)} &  \leq  \frac{\| \psi_4 \|_{\infty}^2}{\sqrt{2 \pi}} \E\!\left[  \int_{\mathbb{R}}  \int_{\mathbb{R}}  \int_{t_1}^{t}  \int_0^1    \es ^{-1/2} [\psi_3'(y+ \xi+ x) 
 - \psi_3'(\xi +y)]^2  \left( \frac{x}{\gamma} \right)^2      \phi_{\gamma^2(s-\es)}(x)  d \gamma ds  dy dx\right]
\\ & =   \frac{\| \psi_4 \|_{\infty}^2}{\sqrt{2 \pi}}   \E\!\left[   \int_{\mathbb{R}}  \int_{\mathbb{R}}  \int_{t_1}^{  t}  \int_0^1    \gamma^{1+2\kappa} \es^{-1/2} (s-\es)^{3/2 + \kappa} \frac{[\psi_3'(y+ \xi+ x) 
 - \psi_3'(\xi+y)]^2}{|x|^{1+2\kappa}} 
 \right.\\ & \left.\qquad \qquad \qquad \qquad\qquad \qquad  \qquad \qquad \qquad \qquad \times  \frac{|x|^{3+ 2 \kappa}}{  (\gamma^2(s-\es))^{3/2 + \kappa}}  \phi_{\gamma^2(s-\es)}(x)      d \gamma ds dy dx\right].
\end{align*}
Setting $w^2=x^2/(\gamma^2(s-\es))$ in \eqref{exp-est} we get that  for all $x \in \mathbb{R}, s \in (0,T], \gamma \in (0,1]$
there exists a constant $c_{3+2\kappa}>0$ such that 
\begin{align} \label{exp_est-W21}
 \frac{|x|^{3+ 2 \kappa}}{  (\gamma^2(s-\es))^{3/2 + \kappa}} \phi_{\gamma^2(s-\es)}(x)  
 &= \frac{|x|^{3+ 2 \kappa}}{  (\gamma^2(s-\es))^{3/2 + \kappa}}  \exp\!\left( -\frac{x^2}{2\gamma^2(s-\es)} \right)\frac{1}{\sqrt{2\pi\gamma^2(s-\es)}} \nonumber \\
  &\le  c_{3+2\kappa} \exp\!\left( -\frac{x^2}{4\gamma^2(s-\es)} \right) \frac{1}{\sqrt{2\pi\gamma^2(s-\es)}}
\\ & \leq c_{3+2\kappa}   \frac{1}{\sqrt{2\pi\gamma^2(s-\es)}}.
\end{align}  
Therefore,
\begin{align*}
 \widetilde{\mathcal{E}}_{1}{ (t)}  & \leq 
 \frac{c_{3+2\kappa} \| \psi_4 \|_{\infty}^2}{2 \pi}     \int_{\mathbb{R}}  \int_{\mathbb{R}}  \int_{t_1}^{t}    \int_0^1    \gamma^{2\kappa} \es^{-1/2} (s-\es)^{1 + \kappa} \frac{\E \, [|\psi'_3(y+ x + \xi ) 
 -  \psi_3'(y + \xi)|^2]}{|x|^{1+2\kappa}} d \gamma ds  dy dx \\
& =  \frac{c_{3+2\kappa} \| \psi_4 \|_{\infty}^2}{2 \pi}     \int_{t_1}^{t}    \int_0^1    \gamma^{2\kappa} \es^{-1/2} (s-\es)^{1 + \kappa}
 \E\!\left[ | \psi'_3(\cdot+\xi)|_{\kappa}^2  \right]
d \gamma ds .
\end{align*}
Corollary \ref{sobolev-2} gives
$$  \E\!\left[ | \psi'_3(\cdot+\xi)|_{\kappa}^2  \right] = |\psi'_3 |_{\kappa}^2, $$
and hence
we obtain for all $t \in [0,T]$,
\begin{align} \label{est-W21-2} 
\widetilde{\mathcal{E}}_{1}{ (t)}  \nonumber
 & \leq    \frac{c_{3+2\kappa} \| \psi_4 \|_{\infty}^2 | \psi'_3 |_{\kappa}^2 }{2 \pi}  \int_{t_1}^{t}   \int_0^1    \gamma^{2\kappa} \es^{-1/2} (s-\es)^{1 + \kappa} d \gamma ds
\\ & \leq  \frac{c_{3+2\kappa} \| \psi_4 \|_{\infty}^2 | \psi'_3 |_{\kappa}^2 }{2 \pi (1+ 2 \kappa)} \| \pi_n \|^{1+\kappa} \int_{t_1}^{T}  \es^{-1/2} ds.
\end{align} 
Combining \eqref{W21W22}, \eqref{est-W22}, \eqref{est-W21-1}, and \eqref{est-W21-2} concludes the proof.
\end{proof}

\subsection{Error analysis of the quadrature problem}

Now we will consider two specific discretizations:~an equidistant discretization $\pi_n^{equi}$ given by
\begin{align}\label{equidisc} t_k^{equi}= T \frac{k}{n}, \qquad  k=0, \ldots, n,\end{align}
and the non-equidistant discretization $\pi_n^{*}$ given by
\begin{align}\label{net-ga}
 t_k^*=T  \left(\frac{k}{n}\right)^{\!2}, \quad k=0, \ldots, n. \end{align}
Clearly, we have
$$  t_{k+1}^*-t_k^* =   \frac{2k+1}{n} \cdot \frac{T}{n},   \qquad k=0, \ldots, n-1,$$
and
\begin{align}\label{Delta-est} \|  \pi_n^* \| = \max_{k=0,\ldots, n-1} |t_{k+1}^*-t_k^*|= \left( 2-  \frac{1}{n} \right)  \cdot \frac{T}{n}\le\frac{2T}{n}. \end{align}
Moreover,  we have:
\begin{lemma}\label{sum-int}
Let $p\in(0,1)$. For $\pi_n^{equi}$ and $\pi_n^*$ we have 
\begin{align*}\sum_{k=1}^{n-1}  t_k^{-p} (t_{k+1}-t_k)
\leq \frac{3}{2}\frac{T^{1-p}}{1-p}.
  \end{align*}
\end{lemma}
\begin{proof}
Consider first $\pi_n^{equi}$.
Using Riemann sums
we obtain
\begin{align*}
\sum_{k=1}^{n-1}  t_k^{-p}   (t_{k+1}-t_k)
= T^{1-p}  \sum_{k=1}^{n-1}  \left( \frac{k}{n}\right)^{\!-p}   \frac{1}{n}
\le
T^{1-p}\int_0^{1} \left(\frac{1}{s}\right)^{p} ds
=\frac{T^{1-p}}{1-p}.
  \end{align*}

 For $\pi_n^{*}$ we have that
\begin{align*}
 \sum_{k=1}^{n-1}  t_k^{-p}   (t_{k+1}-t_k)
& = T^{1-p}  \sum_{k=1}^{n-1}  \left( \frac{k}{n}\right)^{-2p}   \frac{2k+1}{n^2}
\le 3 T^{1-p}  \sum_{k=1}^{n-1}  \left( \frac{k}{n}\right)^{1-2p}   \frac{1}{n}
\\ & \leq 3 T^{1-p}
\int_0^{1} \left(\frac{1}{s}\right)^{2p-1} ds
= \frac{3}{2}\frac{T^{1-p}}{1-p}.
  \end{align*} 
Note the case distinction in $p<1/2$, $p=1/2$, and $p>1/2$ for the Riemann sums.
\end{proof}

Our main result is:

\begin{theorem}\label{main_2}  Let Assumptions \ref{ass} and \ref{ass_b}    hold.  Then there exist  constants $C^{(Q),equi}_{b,T,\kappa}>0$ and $C^{(Q),*}_{b,T,\kappa}>0$  such that 
$$ \sup_{t \in [0,T]} \mathcal{W}^{(\pi_n^{equi})}_{t} \leq   C^{(Q),equi}_{b,T,\kappa} \cdot \left(   \frac{1}{n^{1+\kappa}} + \frac{1}{n^{3/2}} \right)$$
and  
$$ \sup_{t \in [0,T]}  \mathcal{W}^{(\pi_n^{*})}_{t} \leq   C^{(Q),*}_{b,T,\kappa} \cdot \frac{1+\log(n)}{n^{1+\kappa}}. $$
\end{theorem}

\begin{proof} We will start  with an arbitrary discretization and specialize only at the end of the steps to $\pi_n^{equi}$ or $\pi_n^*$, if necessary.
For estimating $\mathcal{W}^{(\pi_n)}$ we use 
that 
\begin{align} \label{est-W2_prem}
\mathcal{W}^{(\pi_n)}_{t} \leq 2 \left(\mathcal{W}_1{(t)}+ \mathcal{W}_2{(t)} \right),
\end{align}
where
\begin{align*}
\mathcal{W}_1{(t)} &= \E\!\left[ \left| \int_0^{t}    \left[(\varphi'b)(W_s+\xi)- (\varphi'b)(W_{\es}+\xi) \right] ds \right|^2\right] , \\ 
\mathcal{W}_2{(t)} &=\E\!\left[
\left| \int_0^{t}     \left[ \varphi'(W_{s}+\xi)  -  \varphi'(W_\es+\xi) \right] b(W_{\es}+\xi) ds \right|^2\right].
\end{align*}

{\it Step 1.} Setting $\psi_3=\varphi'$ and $\psi_4=b$, noting that $\varphi''=-2b \varphi'$, and using Lemma  \ref{sobolev} we obtain that Lemma \ref{step-decomp_2} can be applied to estimate $\mathcal{W}_2$. Thus,  there exists a constant $ c^{(qs)}_{\varphi', b, \kappa, T} >0$ such that
\begin{align*}\sup_{t \in [0,T]}
\mathcal{W}_2{(t)} \leq  c^{(qs)}_{\varphi', b, \kappa, T}  \left( 1 + \sum_{k=1}^{n-1 }t_{k}^{-1/2}(t_{k+1} -t_k)  \right) \| \pi_n \|^{1+\kappa}.
\end{align*}
and using Lemma \ref{sum-int} it follows that for both $\pi_n^{equi}$ and $\pi_n^*$
\begin{align} \label{est-W2} \sup_{t \in [0,T]}
\mathcal{W}_2{(t)} \leq  c^{(qs)}_{\varphi', b, \kappa, T}  \left( 1 + 3 T^{1/2}  \right) \| \pi_n \|^{1+\kappa}.
\end{align}

{\it Step 2.}
For the remaining term, note that
$$ |\varphi'b(\cdot+\xi)|_{\kappa} = |\varphi'b|_{\kappa} < \infty $$
by Corollary \ref{sobolev-2} and
\begin{align}\label{est-W1-1}
\mathcal{W}_1{(t)} & \le 8\|\varphi'b\|_{\infty}^2 ( t_1 +( t-\et) \1_{[t_1,T]}(t))^2
\\&\quad+ 2  \cdot \1_{[t_1,T]}(t)\cdot\E\!\left[ \left| \int_{t_1}^{\et}    \left[(\varphi'b)(W_s+\xi)- (\varphi'b)(W_{\es}+\xi) \right] ds \right|^2 \right]
\\&\le 32\|\varphi'b\|_{\infty}^2 \| \pi_n \|^2 + 2\cdot \1_{[t_1,T]}(t)\cdot\ \E\!\left[ \left| \int_{t_1}^{\et}   \left[(\varphi'b)(W_s+\xi)- (\varphi'b)(W_{\es}+\xi) \right] ds \right|^2 \right].
\end{align}
In the following, let $\et=t_m$ for some $m \in \{2, \ldots, n\}$ and denote  
$$ I^{k,\ell}= \int_{t_k}^{t_{k+1}} \int_{t_\ell}^{t_{\ell+1}}  \left( (\varphi'b)(W_s+\xi)-(\varphi'b)(W_{t_k}+\xi) \right) \left((\varphi'b)(W_t+\xi)-(\varphi'b)(W_{t_{\ell}}+\xi) \right) dt ds . $$ 
We have that
\begin{equation}\label{est-W1-2}
\begin{split}
& \E\!\left[ \left| \int_{t_1}^{\et}    \left[(\varphi'b)(W_s+\xi)- (\varphi'b)(W_{\es}+\xi) \right] ds \right|^2 \right] 
=2 \sum_{k=2}^{{m}-1} \sum_{\ell=1}^{k-1}  \E[I^{k,\ell} ]+  \sum_{k=1}^{{m}-1}\E[I^{k,k} ]
. 
\end{split}
\end{equation}

{\it Step 3.} 
Using $2xy \leq x^2+y^2$ for $x,y\in\R$  we obtain that
\begin{align*}
\sum_{k=1}^{{m}-1}\E[I^{k,k} ]&=\sum_{k=1}^{{m}-1} 
\int_{t_k}^{t_{k+1}} \int_{t_k}^{t_{k+1}} \E \!\big[\left( (\varphi'b)(W_s+\xi)-(\varphi'b)(W_{t_k}+\xi) \right) 
)\\ & \qquad \qquad \qquad \qquad \qquad  \qquad \qquad  \times 
\left((\varphi'b)(W_t+\xi)-(\varphi'b)(W_{t_{k}}+\xi) \right) \big]dt ds  \\
&  \leq   \sum_{k=1}^{{m}-1} 
\int_{t_k}^{t_{k+1}} \int_{t_k}^{t_{k+1}}  \E \!\left[\ \left| (\varphi'b)(W_s+\xi)-(\varphi'b)(W_{t_k}+\xi) \right|^2\right] dt ds .
\end{align*}
For $s\ge t_k \ge t_1$, Lemma \ref{lem_1} shows that there exists a constant $c_\kappa>0$ such that 
\begin{align*}
\sum_{k=1}^{{m}-1}\E[I^{k,k} ]&\leq    
  \sum_{k=1}^{{m}-1}  \int_{t_k}^{t_{k+1}} \int_{t_k}^{t_{k+1}} c_{\kappa} |\varphi'b|_{\kappa}^2  (s-t_k)^{\kappa} t_{k}^{-1/2}dt ds \\
& \le
   c_{\kappa} |\varphi'b|_{\kappa}^2 \| \pi_n \|^{1+\kappa} \sum_{k=1}^{n-1}   t_{k}^{-1/2} (t_{k+1}-t_k)
   .
\end{align*}
Now, Lemma \ref{sum-int} gives
\begin{align}\label{Ikk}
\sum_{k=1}^{{m}-1}\E[I^{k,k} ]&\leq    
   3c_{\kappa}  T^{1/2} |\varphi'b|_{\kappa}^2 \| \pi_n \|^{1+\kappa} 
\end{align}for both discretizations.
It remains to take care of the off-diagonal terms with $k- \ell \geq 1$. 

\smallskip

{\it Step 4.}
Consider the case $m \geq 3$ and  $\ell=k-1 \neq  0$. Again using  $2xy \leq x^2+y^2$ for $x,y\in\R$, Lemma \ref{lem_1}, and Lemma \ref{sum-int} we get that
for both discretizations,
\begin{equation}\label{Ikk-1}
\begin{aligned}
2 \sum_{k=2}^{{m}-1}   \E[I^{k,k-1} ]
&= 2 \sum_{k=2}^{{m}-1} 
\int_{t_k}^{t_{k+1}} \int_{t_{k-1}}^{t_{k}} \E \!\big[ \left( (\varphi'b)(W_s+\xi)-(\varphi'b)(W_{t_k}+\xi) \right)
\\ & \quad \qquad \qquad \qquad \qquad  \qquad \qquad \times
 \left((\varphi'b)(W_t+\xi)-(\varphi'b)(W_{t_{k-1}}+ \xi) \right) \big]dt ds  
\\&  \leq    
 \sum_{k=2}^{{m}-1} 
\int_{t_k}^{t_{k+1}} \int_{t_{k-1}}^{t_{k}} \E \!\left[|(\varphi'b)(W_s+\xi)-(\varphi'b)(W_{t_k}+\xi) |^2  \right]dt ds  
\\&  \qquad +   
\sum_{k=2}^{{m}-1} 
\int_{t_k}^{t_{k+1}} \int_{t_{k-1}}^{t_{k}} \E \!\left[|(\varphi'b)(W_t+\xi)-(\varphi'b)(W_{t_{k-1}}+\xi) |^2  \right]dt ds  
\\&  \leq    
\sum_{k=2}^{{m}-1} 
\int_{t_k}^{t_{k+1}} \int_{t_{k-1}}^{t_{k}}     c_{\kappa} |\varphi'b|_{\kappa}^2  (s-t_k)^{\kappa} t_{k}^{-1/2}    dt ds  
\\ & \qquad + \sum_{k=2}^{{m}-1} 
\int_{t_k}^{t_{k+1}} \int_{t_{k-1}}^{t_{k}}     c_{\kappa} |\varphi'b|_{\kappa}^2  (t-t_{k-1})^{\kappa} t_{k-1}^{-1/2}    dt ds  
   \\&\le
  2 c_{\kappa} |\varphi'b|_{\kappa}^2 \| \pi_n \|^{1+\kappa} \sum_{k=1}^{n-1}   t_{k}^{-1/2} (t_{k+1}-t_k)
\leq 
6 c_{\kappa} T^{1/2} |\varphi'b|_{\kappa}^2 \| \pi_n\|^{1+\kappa}  .
\end{aligned}
\end{equation}
\smallskip

{\it Step 5.}
Consider the case {$m \geq 4$,} assume $k \geq \ell +2$, and use \eqref{pts}. We get
\begin{align}\label{Ilk-100}
\E[ I^{k,\ell} | \mathcal{F}_0 ]&=  \int_{\mathbb{R}} \int_{\mathbb{R}} (\varphi'b)(x+\xi)(\varphi'b)(y+\xi) 
\\&\qquad \,\,  \times 
\int_{t_k}^{t_{k+1}} \int_{t_\ell}^{t_{\ell+1}} \left( p_{s,t}(x,y)-p_{t_k,t}(x,y)-p_{s,t_{\ell}}(x,y)+p_{t_k,t_{\ell}}(x,y)\right)dt ds \, dx dy.
\end{align}
First note that
\begin{align}\label{dens}   p_{s,t}(x,y)-p_{t_k,t}(x,y)-p_{s,t_{\ell}}(x,y)+p_{t_k,t_{\ell}}(x,y) = \int_{t_k}^s \int_{t_{\ell}}^t \frac{\partial^2}{\partial u \partial v} p_{u,v}(x,y) \,  dv du.  \end{align}
Now observe that
\begin{align*}
&  \int_{\mathbb{R}} \int_{\mathbb{R}} (\varphi'b)(x+\xi)^2   \left( p_{s,t}(x,y)-p_{t_k,t}(x,y)-p_{s,t_{\ell}}(x,y)+p_{t_k,t_{\ell}}(x,y)\right)dx dy 
\\ & \quad =   \left(\E\!\left[|(\varphi' b)(W_s+\xi)|^2|\cF_0\right]-\E\!\left[|(\varphi' b)(W_{t_k}+\xi)|^2|\cF_0\right]\right)
\\&\quad \quad \quad \quad \quad \quad \quad \quad -
\left(\E\!\left[|(\varphi' b)(W_s+\xi)|^2|\cF_0\right]-\E\!\left[|(\varphi' b)(W_{t_k}+\xi)|^2|\cF_0\right]\right) =0
\end{align*}
and analogously
\begin{align*}
   \int_{\mathbb{R}} \int_{\mathbb{R}} (\varphi'b)(y+\xi)^2   \left( p_{s,t}(x,y)-p_{t_k,t}(x,y)-p_{s,t_{\ell}}(x,y)+p_{t_k,t_{\ell}}(x,y)\right)  dx dy  = 0.
\end{align*}
Combining this with \eqref{Ilk-100} and \eqref{dens} we obtain
\begin{align*}
  \E [I^{k,\ell} | \mathcal{F}_0] & =   -\frac{1}{2} \int_{\mathbb{R}} \int_{\mathbb{R}} |(\varphi'b)(x+\xi)-(\varphi'b)(y+\xi)|^2 \\ &  \qquad \qquad \qquad \times \int_{t_k}^{t_{k+1}} \int_{t_{\ell}}^{t_{\ell+1}} \int_{t_k}^s \int_{t_{\ell}}^t \frac{\partial^2}{\partial u \partial v} p_{u,v}(x,y) \,  dv du \,  dt ds\,  dx dy
\\ & = -\frac{1}{2} \int_{\mathbb{R}} \int_{\mathbb{R}} \frac{|(\varphi'b)(x+\xi)-(\varphi'b)(y+\xi)|^2}{|x-y|^{1+2\kappa}} 
\\&\quad \quad \quad \quad \quad \quad \times
\int_{t_k}^{t_{k+1}} \int_{t_{\ell}}^{t_{\ell+1}} \int_{t_k}^s \int_{t_{\ell}}^t |x-y|^{1+2\kappa} \frac{\partial^2}{\partial u \partial v} p_{u,v}(x,y) \,  dv du \,  dt ds\,  dx dy
.
\end{align*}
Corollary \ref{sobolev-2} and Lemma  \ref{lem_pab} ensure that there exists a constant $\densityconst>0$ such that
\begin{align*}
  \E [I^{k,\ell} | \mathcal{F}_0] &  \leq  \frac{\densityconst}{2} |\varphi' b|_\kappa^2 \int_{t_k}^{t_{k+1}} \int_{t_{\ell}}^{t_{\ell+1}} \int_{t_k}^s \int_{t_{\ell}}^t
 \left( |u-v|^{ \kappa -2}v^{-1/2} +  |u-v|^{\kappa -1} v^{-3/2} \right)
 dv du  dt ds
 \\&\le
  \frac{\densityconst}{2} |\varphi' b|_\kappa^2 (t_{k+1}-t_k)(t_{\ell+1}-t_\ell)\int_{t_k}^{t_{k+1}} \int_{t_{\ell}}^{t_{\ell+1}}
 \left( |u-v|^{ \kappa -2}v^{-1/2} +  |u-v|^{\kappa -1} v^{-3/2} \right)
 dv du.
\end{align*}
Hence,
\begin{equation}\label{Ilk-1}
\begin{aligned}
 &2 \sum_{k=3}^{{m}-1} \sum_{\ell=1}^{k-2}  \E[I^{k,\ell} ]
 =
  2 \sum_{k=3}^{{m}-1} \sum_{\ell=1}^{k-2} \E[ \E[I^{k,\ell} |\cF_0] ]
  \\&\qquad\le \densityconst |\varphi' b|_\kappa^2  \sum_{k=3}^{n-1} \sum_{\ell=1}^{k-2} (t_{k+1}-t_k)(t_{\ell+1}-t_\ell) \\ & \qquad \qquad \qquad \qquad \qquad \qquad \times \int_{t_k}^{t_{k+1}} \int_{t_{\ell}}^{t_{\ell+1}}
 \left( |u-v|^{ \kappa -2}v^{-1/2} +  |u-v|^{\kappa -1} v^{-3/2} \right)
 dv du.
\end{aligned}
\end{equation}
Summarizing the above estimates \eqref{est-W2_prem}, \eqref{est-W2}, \eqref{est-W1-1}, \eqref{est-W1-2}, \eqref{Ikk}, \eqref{Ikk-1}, and \eqref{Ilk-1}, establishes for all $t\in[0,T]$ that
\begin{equation} \label{sum_est_before_disc}
\begin{aligned}
\mathcal{W}_t^{(\pi_n)} & \leq  \left(2( 1+3T^{1/2} )c^{(qs)}_{\varphi', b, \kappa, T}+64 \|\varphi' b\|^2_\infty {\|\pi_n\|^{1-\kappa}}+ 36 c_\kappa T^{1/2}|\varphi'b|_\kappa^2\right) \| \pi_n \|^{1+\kappa} \\ & \quad  + 4\densityconst |\varphi'b|_\kappa^2 \sum_{k=3}^{n-1} \sum_{\ell=1}^{k-2} (t_{k+1}-t_k)(t_{\ell+1}-t_\ell)  
\\&\qquad \quad \times
\int_{t_k}^{t_{k+1}} \int_{t_{\ell}}^{t_{\ell+1}}
 \left( |u-v|^{ \kappa -2}v^{-1/2} +  |u-v|^{\kappa -1} v^{-3/2} \right) dvdu.
\end{aligned}
\end{equation}

\smallskip

{\it Step 6, Case 1.}
First consider the non-equidistant discretization \eqref{net-ga}.
Observe that
\begin{align} \label{part_int}
 2 (1-\kappa)\int_{t_{\ell}}^{t_{\ell+1}}
 |u-v|^{\kappa -2} v^{-1/2} \, dv =
  \int_{t_{\ell}}^{t_{\ell+1}}
 |u-v|^{\kappa -1} v^{-3/2} 
 \,  dv + 2|u-v|^{\kappa-1} v^{-1/2} \big{|}_{v=t_{\ell}}^{v=t_{\ell+1}}
\end{align} for $u \geq t_{\ell+1}$.
Thus we have
\begin{equation} \label{split_*}
\begin{aligned}
& \sum_{k=3}^{n-1} \sum_{\ell=1}^{k-2} (t_{k+1}-t_k)(t_{\ell+1}-t_\ell)  \int_{t_k}^{t_{k+1}} \int_{t_{\ell}}^{t_{\ell+1}}
 \left( |u-v|^{ \kappa -2}v^{-1/2} +  |u-v|^{\kappa -1} v^{-3/2} \right) dvdu \\& \quad = \left( 1+\frac{1}{2(1-\kappa)} \right) \mathcal{I}_n^{*,(1)} + \frac{1}{1-\kappa} \mathcal{I}_n^{*,(2)}
\end{aligned}
\end{equation}
with
\begin{align*}
\mathcal{I}_n^{*,(1)}& = \sum_{k=3}^{n-1} \sum_{\ell=1}^{k-2} (t_{k+1}-t_k)(t_{\ell+1}-t_\ell)   \int_{t_k}^{t_{k+1}} \int_{t_{\ell}}^{t_{\ell+1}} |u-v|^{\kappa -1} v^{-3/2} dvdu,   \\  \mathcal{I}_n^{*,(2)}&= \sum_{k=3}^{n-1} \sum_{\ell=1}^{k-2} (t_{k+1}-t_k)(t_{\ell+1}-t_\ell)
 \int_{t_k}^{t_{k+1}}\left( |u-t_{\ell+1}|^{\kappa-1} t_{\ell+1}^{-1/2}- |u-t_{\ell}|^{\kappa-1} t_{\ell}^{-1/2} \right) \, du . 
\end{align*}

Since $k\ge \ell+2$, $\kappa\in(0,1)$ and $x^{\kappa}-y^{\kappa} \leq |x-y|^{\kappa}$ for $x>y\geq 0$ we have
\begin{align*}
&  \int_{t_k}^{t_{k+1}} \int_{t_{\ell}}^{t_{\ell+1}}  
  |u-v|^{\kappa -1} v^{-3/2} 
 \,  dv du \\ & \qquad \leq   \int_{t_k}^{t_{k+1}} \int_{t_{\ell}}^{t_{\ell+1}}  
  |u-t_{\ell+1}|^{\kappa -1} v^{-3/2} \,dv  du \\ &  \qquad = 2 \left( t_{\ell}^{-1/2} -t_{\ell+1}^{-1/2} \right) \int_{t_k}^{t_{k+1}} 
  |u-t_{\ell+1}|^{\kappa -1}du  \\
&  \qquad =
\frac{2}{\kappa} \left( t_{\ell}^{-1/2} -t_{\ell+1}^{-1/2} \right) \left( |t_{k+1}-t_{\ell+1}|^{\kappa }-
  |t_k-t_{\ell+1}|^{\kappa }\right) 
     \\& \qquad \le
\frac{2}{\kappa} \left( t_{\ell}^{-1/2} -t_{\ell+1}^{-1/2} \right)  |t_{k+1}-t_k|^{\kappa}
  .  
\end{align*}
 Thus it follows
\begin{align*} 
&  \sum_{k=3}^{n-1} \sum_{\ell=1}^{k-2} (t_{k+1}-t_k) (t_{\ell+1}-t_{\ell}) \int_{t_k}^{t_{k+1}} \int_{t_{\ell}}^{t_{\ell+1}}  
|u-v|^{\kappa -1} v^{-3/2}  
 \,  dv du  \nonumber 
\\ & \le \frac{2}{\kappa} \| \pi_n^*\|^{\kappa}
 \sum_{k=3}^{n-1} \sum_{\ell=1}^{k-2}   (t_{k+1}-t_k) (t_{\ell+1}-t_{\ell})   \left( t_{\ell}^{-1/2} -t_{\ell+1}^{-1/2} \right) 
\\ & \le \frac{2T}{\kappa} \| \pi_n^*\|^{\kappa} \sum_{\ell=1}^{n-3} (t_{\ell+1}-t_{\ell})   \left( t_{\ell}^{-1/2} -t_{\ell+1}^{-1/2} \right) 
\\ &  =  \frac{2T^{3/2}}{\kappa} \| \pi_n^*\|^{\kappa}
 \sum_{\ell=1}^{n-3}   \frac{2\ell+1}{n^2} \left( \frac{n}{\ell} - \frac{n}{\ell+1} \right)
\\ &  =  \frac{2T^{3/2}}{\kappa} \| \pi_n^*\|^{\kappa} \frac{1}{n} \sum_{\ell=1}^{n-3}   \frac{2\ell +1}{\ell(\ell+1)}
\\ &  \leq  \frac{2^{2+\kappa}T^{3/2+\kappa}}{\kappa}  \frac{1}{n^{1+ \kappa}} \sum_{\ell=1}^{n-3}   \frac{1}{\ell},
 \end{align*} 
where we have used \eqref{net-ga}, \eqref{Delta-est}, and that $t_{\ell+1}-t_{\ell}=T(2\ell +1)n^{-2}$.
Since 
$$  \sum_{\ell=1}^{n}   \frac{1}{\ell} \leq 1+ \log(n), $$ we have
\begin{align}  \label{Ilk-2}
   \mathcal{I}_n^{*,(1)}
 \leq      \frac{2^{2+\kappa}T^{3/2+\kappa}}{\kappa}  \frac{1+\log(n)}{n^{1+ \kappa}}.
\end{align}
So, the remaining term to estimate is 
\begin{align*}
\mathcal{I}_n^{*,(2)}& =   \sum_{k=3}^{n-1} \sum_{\ell=1}^{k-2}  (t_{k+1}-t_k) (t_{\ell+1}-t_{\ell}) \int_{t_k}^{t_{k+1}}\left( |u-t_{\ell+1}|^{\kappa-1} t_{\ell+1}^{-1/2}- |u-t_{\ell}|^{\kappa-1} t_{\ell}^{-1/2} \right) \, du . 
\end{align*}
We get
\begin{align*}
\mathcal{I}_n^{*,(2)} & =
   \sum_{\ell=1}^{n-3} \sum_{k=\ell+2}^{n-1}  (t_{k+1}-t_k) (t_{\ell+1}-t_{\ell}) \int_{t_k}^{t_{k+1}}\left( |u-t_{\ell+1}|^{\kappa-1} t_{\ell+1}^{-1/2}- |u-t_{\ell}|^{\kappa-1} t_{\ell}^{-1/2} \right) \, du
    \\ & \leq 
   \| \pi_n^*\| \sum_{\ell=1}^{n-3}  (t_{\ell+1}-t_{\ell}) \int_{t_{\ell+2}}^{T}\left( |u-t_{\ell+1}|^{\kappa-1} t_{\ell+1}^{-1/2}- |u-t_{\ell}|^{\kappa-1} t_{\ell}^{-1/2} \right) \, du
\\ & =
\frac{   \| \pi_n^*\| }{\kappa} 
   \sum_{\ell=1}^{n-3}  (t_{\ell+1}-t_{\ell}) \left[\left( |T-t_{\ell+1}|^{\kappa} -|t_{\ell+2}-t_{\ell+1}|^{\kappa} \right) t_{\ell+1}^{-1/2} - \left( |T-t_{\ell}|^{\kappa} -|t_{\ell+2}-t_{\ell}|^{\kappa} \right) t_{\ell}^{-1/2}\right].
\end{align*}
Using \eqref{net-ga}, \eqref{Delta-est}, and estimating negative terms from above by zero we obtain that
\begin{align*}
\mathcal{I}_n^{*,(2)}& \leq
\frac{   \| \pi_n^*\| }{\kappa} 
   \sum_{\ell=1}^{n-3} (t_{\ell+1}-t_{\ell}) \left[\left( |T-t_{\ell+1}|^{\kappa} - |T-t_{\ell}|^{\kappa} \right) t_{\ell+1}^{-1/2}  + |t_{\ell+2}-t_{\ell}|^{\kappa}  t_{\ell}^{-1/2}\right] 
   \\ &  \leq  \frac{   \| \pi_n^*\| }{\kappa} 
   \sum_{\ell=1}^{n-3}   (t_{\ell+1}-t_{\ell})  |t_{\ell+2}-t_{\ell}|^{\kappa}  t_{\ell}^{-1/2} 
  \\ &  \leq  \frac{   2^{\kappa}\| \pi_n^*\|^{1+\kappa} }{\kappa} 
   \sum_{\ell=1}^{n-3}    (t_{\ell+1}-t_{\ell})  t_{\ell}^{-1/2} 
 \\ &  \leq  \frac{   2^{1+2\kappa} T^{1+\kappa} }{\kappa}  \frac{1}{n^{1+\kappa}}
   \sum_{\ell=1}^{n-3}   (t_{\ell+1}-t_{\ell})    t_{\ell}^{-1/2} 
.
\end{align*}
Finally, Lemma \ref{sum-int} establishes 
\begin{align}\label{lastcase1} \mathcal{I}_n^{*,(2)} \leq  \frac{   2^{1+2\kappa}3 T^{3/2+\kappa} }{\kappa}  \frac{1}{n^{1+\kappa}}.
\end{align}
Combining \eqref{sum_est_before_disc} with  \eqref{split_*}, \eqref{Ilk-2}, and \eqref{lastcase1} finishes the analysis of $\mathcal{W}_t^{(\pi_n^*)}$ .\\

\medskip

{\it Step 6, Case 2.}
Now consider the equidistant discretization \eqref{equidisc}.
We make use of \eqref{part_int} in a different way than above. It holds that
\begin{equation} \label{split_equi}
\begin{aligned}
& \sum_{k=3}^{n-1} \sum_{\ell=1}^{k-2} (t_{k+1}-t_k)(t_{\ell+1}-t_\ell)  \int_{t_k}^{t_{k+1}} \int_{t_{\ell}}^{t_{\ell+1}}
 \left( |u-v|^{ \kappa -2}v^{-1/2} +  |u-v|^{\kappa -1} v^{-3/2} \right) dvdu \\& \quad = \left( 3-2\kappa \right) \mathcal{I}_n^{equi,(1)} -2 \mathcal{I}_n^{equi,(2)}
\end{aligned}
\end{equation}
with
\begin{align*}
\mathcal{I}_n^{equi,(1)}& = \frac{T^2}{n^2} \sum_{k=3}^{n-1} \sum_{\ell=1}^{k-2}   \int_{t_k}^{t_{k+1}} \int_{t_{\ell}}^{t_{\ell+1}} |u-v|^{\kappa -2} v^{-1/2} dvdu,   \\  \mathcal{I}_n^{equi,(2)}&= \frac{T^2}{n^2}   \sum_{k=3}^{n-1} \sum_{\ell=1}^{k-2} 
 \int_{t_k}^{t_{k+1}}\left( |u-t_{\ell+1}|^{\kappa-1} t_{\ell+1}^{-1/2}- |u-t_{\ell}|^{\kappa-1} t_{\ell}^{-1/2} \right) \, du . 
\end{align*}
Exploiting the telescoping sum in the second term, we get
\begin{align*}
\mathcal{I}_n^{equi,(2)} & = \frac{T^2}{n^2}
   \sum_{k=3}^{n-1} \int_{t_k}^{t_{k+1}}\left( |u-t_{k-1}|^{\kappa-1} t_{k-1}^{-1/2}- |u-t_{1}|^{\kappa-1} t_{1}^{-1/2} \right) \, du
    \\ & \geq -  \frac{T^2}{n^2}
  \int_{t_1}^{T} |u-t_{1}|^{\kappa-1} t_{1}^{-1/2}  \, du
\\ & \geq -  \frac{T^{3/2 + \kappa}}{\kappa} \frac{1}{n^{3/2}},
\end{align*}
since $t_1=T/n$.
It follows that
\begin{align}\label{est-equi-2}
 -2 \mathcal{I}_n^{equi,(2)} \leq \frac{2T^{3/2 + \kappa}}{\kappa} \frac{1}{n^{3/2}}.
\end{align}
Moreover, we have
\begin{align*}
\int_{t_k}^{t_{k+1}} \int_{t_{\ell}}^{t_{\ell+1}} |u-v|^{\kappa -2} v^{-1/2} dvdu  & \leq \frac{T}{n}  \int_{t_{\ell}}^{t_{\ell+1}} |t_k-v|^{\kappa -2} v^{-1/2} dv  \\ &\leq 
 \frac{T}{n}|t_k-t_{\ell+1}|^{\kappa -2}  \int_{t_{\ell}}^{t_{\ell+1}} v^{-1/2} dv
 \\ &= \frac{T^{\kappa-1}}{n^{\kappa-1}} |k-\ell-1|^{\kappa-2}   \int_{t_{\ell}}^{t_{\ell+1}} v^{-1/2} dv .
\end{align*}
Thus, we end up with
\begin{equation}\label{lastcase2}
\begin{aligned}
\mathcal{I}_n^{equi,(1)}&  \leq \frac{T^{1+\kappa}}{n^{1+\kappa}} \sum_{k=3}^{n-1} \sum_{\ell=1}^{k-2}  |k-\ell-1|^{\kappa-2}   \int_{t_{\ell}}^{t_{\ell+1}} v^{-1/2} dv
\\ & = \frac{T^{1+\kappa}}{n^{1+\kappa}} \sum_{\ell=1}^{n-3} \int_{t_{\ell}}^{t_{\ell+1}} v^{-1/2} dv  \sum_{k=\ell+2}^{n-1}  |k-\ell-1|^{\kappa-2}  
\\ &  \leq \frac{T^{1+\kappa}}{n^{1+\kappa}}  \int_{0}^{T} v^{-1/2} dv \, \sum_{j=1}^{n-3}   j^{\kappa-2}  
\\ & \leq  \frac{2T^{3/2+\kappa}}{n^{1+\kappa}} \sum_{j=1}^{n}   j^{\kappa-2},
\end{aligned}
\end{equation}
where  $\kappa \in (0,1)$ implies 
$  \sum_{j=1}^{\infty}   j^{\kappa-2}< \infty$.

Combining  \eqref{sum_est_before_disc} with \eqref{split_equi}, \eqref{est-equi-2}, and \eqref{lastcase2} finishes the analysis of $\mathcal{W}_t^{(\pi_n^{equi})}$ and the proof of this theorem.
\end{proof}

\begin{remark} If the initial condition $\xi$  has additionally a bounded Lebesgue density,  then  
 \citet[Theorem 8]{altmeyer2017} yields for the term $\mathcal{W}_1(T) $ in the proof  of Theorem \ref{main_2}
the convergence order $(1+\kappa)/2$ also for equidistant discretizations,  i.e.~there is no cut-off of the convergence order for $\kappa \in [1/2,1)$. Due to the independence of $\xi$ and $W$, the assumption of a bounded Lebesgue density $ \varsigma$ for $\mathbb{P}^{\xi}$  leads to a smoothing effect in the integration problem;  roughly spoken, $(\varphi'b)(\cdot+\xi)$ can be replaced by the convolution $ \int_{\mathbb{R}} (\varphi'b)(\cdot + z) \varsigma(z) dz$.
\end{remark}

We finally obtain the following statement for the convergence rate of the EM schemes $x^{(\pi_n^{equi})}$ and $x^{(\pi_n^{*})}$.

\begin{corollary}\label{cor_em}
Let Assumptions \ref{ass} and \ref{ass_b}   hold.  Then, for all $\epsilon\in(0,1)$ there exist  constants $C^{(EM),equi}_{\epsilon, \mu,T,\kappa}>0$ and  $C^{(EM),*}_{\epsilon, \mu,T,\kappa}>0$ such that 
$$ \sup_{t \in [0,T]} \E\!\left[ \left|X_t-x_t^{(\pi_n^{equi})} \right|^2\right] \leq  C^{(EM),equi}_{\epsilon, \mu,T,\kappa} \cdot \left(   \frac{1}{n^{1+\kappa-\epsilon}} + \frac{1}{n^{3/2-\epsilon}} \right)$$
and
$$  \sup_{t \in [0,T]}   \E\!\left[ \left|X_t-x_t^{(\pi_n^{*})}\right|^2\right] \leq  C^{(EM),*}_{\epsilon, \mu,T,\kappa} \cdot  \frac{1}{n^{1+\kappa-\epsilon}}.$$
\end{corollary}

\begin{proof}
Theorems \ref{main_1} and \ref{main_2} yield that there exist constants $C^{(R)}_{\varepsilon,a,b,T},C^{(Q),*}_{b,T,\kappa}>0$ such that
\begin{align*}
\sup_{t \in [0,T]}   \E\!\left[ \left|X_t-x_t^{(\pi_n^{*})}\right|^2\right] &\leq  C^{(R)}_{\varepsilon, a,b,T}   \left[\|\pi_n^*\|^2 +\left(C^{(Q),*}_{b,T,\kappa}   \frac{1+\log(n)}{n^{1+\kappa}}\right)^{1-\varepsilon}  \right]
\\&\le
  C^{(R)}_{\varepsilon, a,b,T}   \left[\frac{4T^2}{n^2} +\left(C^{(Q),*}_{b,T,\kappa} \frac{1+\log(n)}{n^{1+\kappa}}\right)^{1-\varepsilon}  \right]
,
\end{align*}
where we used \eqref{Delta-est}.
The estimate for the equidistant discretization is obtained analogously.
\end{proof}

\section*{Acknowledgements}

The authors are very thankful to the referees for their insightful comments and remarks.

M.~Sz\"olgyenyi has been supported by the AXA Research Fund grant `Numerical Methods for Stochastic Differential Equations with Irregular Coefficients with Applications in Risk Theory and Mathematical Finance'.
A part of this article was written while M.~Sz\"olgyenyi was affiliated with the Seminar for Applied Mathematics and the RiskLab Switzerland, ETH Zurich, R\"amistrasse 101, 8092 Zurich, Switzerland.

\end{document}